\documentclass{article}%
\usepackage{amsmath}%
\usepackage{amsfonts}%
\usepackage{amssymb}%

\usepackage{amsthm}

\usepackage{graphicx}
\usepackage{todonotes}
\usepackage{a4wide}
\usepackage{url}
\usepackage[utf8]{inputenc}
\usepackage{hyperref}
\newtheorem{theorem}{Theorem}
\numberwithin{theorem}{section}

\newtheorem{assumption}[theorem]{Assumption}

\newtheorem{condition}[theorem]{Condition}

\newtheorem{corollary}[theorem]{Corollary}

\newtheorem{definition}[theorem]{Definition}

\newtheorem{lemma}[theorem]{Lemma}

\newtheorem{proposition}[theorem]{Proposition}

\theoremstyle{definition}
\newtheorem{remark}[theorem]{Remark}

\newcommand{\abGal}[1] {\operatorname{Gal}\big(\overline{#1}/#1\big)}

\usepackage{amsmath}
\begin{document}

\title{Pink-type results for general subgroups of $\operatorname{SL}_2(\mathbb{Z}_\ell)^n$}
\author{Davide Lombardo}
\date{}
\maketitle

\begin{abstract}
We study open subgroups $G$ of $\operatorname{SL}_2(\mathbb{Z}_\ell)^n$ in terms of some associated Lie algebras without assuming that $G$ is a pro-$\ell$ group, thereby extending a theorem of Pink. The result has applications to the study of families of Galois representations.
\end{abstract}
\section{Motivation and statement of the result}
The ultimate goal of this work is the study of the images of certain Galois representations with values in $\operatorname{GL}_2(\mathbb{Z}_\ell)^n$, such as those afforded by the Tate modules of elliptic curves, or those arising from modular forms. In view of this aim, one would wish to provide a manageable description of these images; however, it turns out that it is beneficial, and in a sense simpler, to consider arbitrary subgroups $G$ of $\operatorname{GL}_2(\mathbb{Z}_\ell)^n$ without making any reference to their origin, and in the present work Galois representations will play virtually no role. In most applications to the study of Galois representations, the main object of interest is actually the intersection $G \cap \operatorname{SL}_2(\mathbb{Z}_\ell)^n$, and furthermore it is an easy matter to pass from results on subgroups of $\operatorname{SL}_2(\mathbb{Z}_\ell)^n$ to results on subgroups of $\operatorname{GL}_2(\mathbb{Z}_\ell)^n$, so we shall actually mostly work with subgroups $G$ of $\operatorname{SL}_2(\mathbb{Z}_\ell)^n$. Any such $G$ is the extension of a `finite' part, the image $G(\ell)$ of the reduction $G \to \operatorname{SL}_2(\mathbb{F}_\ell)^n$, by a `Lie' part, the kernel of this reduction.

When $G$ is closed and $G(\ell)$ is trivial (or more generally when $G$ is pro-$\ell$), and $\ell$ is odd, a construction due to Pink \cite{MR1241950} gives a very concrete and handy description of $G$ in terms of a certain $\mathbb{Z}_\ell$-Lie algebra $\mathcal{L}(G)$ (together with some additional data which is not very important to our present discussion). Furthermore, if $G$ is the image of a representation of $\abGal{K}$ ($K$ a number field), the condition that $G(\ell)$ be trivial can always be met by replacing $K$ by a finite extension, so that Pink's theorem applies. Note however that the degree of this extension depends on $\ell$: while this is often perfectly fine when considering a single Galois representation, it may become a major drawback when dealing with infinite families $G_\ell$ indexed by the rational primes (as is the case, for example, with the action of $\abGal{K}$ on the various Tate modules of an abelian variety).
Furthermore, Pink's theorem does not apply to $\ell=2$, which might again be quite a hindrance when trying to study the whole system $G_\ell$ at once.

While we cannot hope to give a complete description of $G$ in terms of Pink's Lie algebras when $G$ is not pro-$\ell$, we could try and settle for less, namely a result of the form `when $\mathcal{L}(G)$ contains a large neighbourhood of the identity (given explicitly), we can explicitly find a neighbourhood of the identity of $\operatorname{SL}_2(\mathbb{Z}_\ell)^n$ that is included in $G$'. Note that when dealing with Galois representations we are often interested in `large image' results, for which this weaker form of Pink's theorem would still be adequate. Unfortunately, even this is not possible (cf. for example \cite[§ 4.5]{AdelicEC}), and the best we can hope for is for such a statement to hold not quite for $G$, but for a subgroup $H$ of $G$ such that the index $[G:H]$ is bounded \textit{by a function of $n$ alone}. 
In order to give a concrete statement we shall need some preliminary definitions:

\begin{definition}
For a prime $\ell$ and a positive integer $s$ we let $\mathcal{B}_\ell(s)$ be the open subgroup of $\operatorname{SL}_2(\mathbb{Z}_\ell)$ given by
\[
\left\{ x \in \operatorname{SL}_2(\mathbb{Z}_\ell) \bigm\vert x \equiv \operatorname{Id} \pmod{\ell^s} \right\}.
\]

We also set $\mathcal{B}_\ell(0)=\operatorname{SL}_2(\mathbb{Z}_\ell)$, and for non-negative integers $k_1, \ldots, k_n$ we denote by $\mathcal{B}_\ell(k_1,\ldots, k_n)$ the open subgroup $\prod_{j=1}^n \mathcal{B}_\ell(k_j)$ of $\operatorname{SL}_2(\mathbb{Z}_\ell)^n$.
\end{definition}

\begin{definition}{(cf. \cite{MR1241950})}
Let $\ell$ be a prime, $n$ a positive integer and $G$ be an open subgroup of $\operatorname{GL}_2(\mathbb{Z}_\ell)^n$. If $\ell=2$, assume further that the reduction modulo 4 of $G$ is trivial. Writing elements of $G$ as $n$-tuples $(g_1,\ldots,g_n)$ of elements of $\operatorname{GL}_2(\mathbb{Z}_\ell)$, we define a map $\Theta_n$ by the formula
\[
\begin{array}{cccc}
\Theta_n : & G  & \to & \bigoplus_{i=1}^n \mathfrak{sl}_2(\mathbb{Z}_\ell)\\
& (g_1,\ldots,g_n) & \mapsto & \left( g_1 -\frac{1}{2} \operatorname{tr}(g_1), \ldots, g_n -\frac{1}{2} \operatorname{tr}(g_n)  \right),
\end{array}
\]
and we let $\mathcal{L}(G) \subseteq \mathfrak{sl}_2(\mathbb{Z}_\ell)^n$ be the $\mathbb{Z}_\ell$-span of $\Theta_n(G)$. We call $\mathcal{L}(G)$ the Lie algebra of $G$.
\end{definition}

\begin{theorem}\label{thm_GeneralPink}
Let $\ell$ be an odd prime, $n$ be an integer, and $G$ be an open subgroup of $\operatorname{SL}_2(\mathbb{Z}_\ell)^n$. There exists an open subgroup $H$ of $G$, of index at most $24^n 48^{n(n-1)}$, with the following property: if $\mathcal{L}(H)$ contains $\bigoplus_{i=1}^n \ell^k \mathfrak{sl}_2(\mathbb{Z}_\ell)$ for a certain integer $k>0$, then $H$ contains $B_\ell(p, \ldots, p)$ for $p=80(\max\{n,2\}-1)k$.

Similarly, let $n$ be a positive integer and $G$ be an open subgroup of $\operatorname{SL}_2(\mathbb{Z}_2)^n$. There exists an open subgroup $H$ of $G$ that satisfies $[G:H] \bigm\vert 96^{n}$, is trivial modulo 4 (so that $\mathcal{L}(H)$ is defined), and has the following property: if $\mathcal{L}(H)$ contains $\bigoplus_{i=1}^n 2^k \mathfrak{sl}_2(\mathbb{Z}_2)$ for a certain integer $k>0$, then $H$ contains $B_2(p, \ldots, p)$ for $p=607(\max\{n,2\}-1)k$.

\end{theorem}

While it is certainly true that both this theorem and its proof are quite technical, it should be remarked that this statement does enable us to show exactly the kind of `large image' results we alluded to: the case $n=1$ has been used in \cite{AdelicEC} to show an explicit open image theorem for elliptic curves (without complex multiplication), and in \cite{ProductsEC} we apply the case $n=2$ to extend this result to arbitrary products of non-CM elliptic curves.

A few more words on the proof of theorem \ref{thm_GeneralPink}: as it will be clear from section \ref{sect_RedTo2}, the crucial cases are $n=1$ and $n=2$. While the former has essentially been proven in \cite{AdelicEC}, the latter forms the core of the present paper, and we shall actually prove it in a slightly more precise form than strictly necessary to establish theorem \ref{thm_GeneralPink}. This will be done in sections \ref{sect_Pinkl} and \ref{sect_Pink2} below, where we also give analogous statements for $\operatorname{GL}_2(\mathbb{Z}_\ell)^2$.

\medskip

\noindent\textbf{Notation.} Throughout the paper $\ell$ is a fixed prime number. If $G$ is a closed subgroup of $\operatorname{GL}_2(\mathbb{Z}_\ell)^n$ and $k$ is a positive integer, we denote by $G(\ell^k)$ the image of the projection $G \to \operatorname{GL}_2(\mathbb{Z}/\ell^k\mathbb{Z})^n$. If $H$ is a subgroup of $\operatorname{SL}_2(\mathbb{Z}_\ell)^n$  (resp.~of $\operatorname{SL}_2(\mathbb{F}_\ell)^n$), we denote by $N(H)$ the largest normal subgroup of $H$ which is pro-$\ell$ (resp.~an $\ell$-group); this object is well-defined by lemma \ref{lemma_NIsWellDefined} below. If $x$ is an element of $\operatorname{GL}_2(\mathbb{Z}_\ell)$ (resp.~of $\mathbb{Z}_\ell$), we write $[x]$ for its image in $\operatorname{GL}_2(\mathbb{F}_\ell)$ (resp.~in $\mathbb{F}_\ell$). Since special care is needed to treat the case $\ell = 2$, to write uniform statements we set $v$ either $0$ or $1$ according to whether $\ell$ is odd or $\ell = 2$.
Finally, it will also be useful to introduce some standard elements of $\operatorname{SL}_2(\mathbb{Z}_\ell)$:
for every $a \in \mathbb{Z}_\ell$ we set
\[
L(a)=\left(\begin{matrix} 1 & 0 \\ a & 1 \end{matrix} \right), \;  D(a)=\left(\begin{matrix} 1+a & 0 \\ 0 & (1+a)^{-1} \end{matrix} \right), \; R(a)=\left(\begin{matrix} 1 & a \\ 0 & 1 \end{matrix} \right).
\]
Notice that $L(a), R(a)$ belong to $\operatorname{SL}_2(\mathbb{Z}_\ell)$ for any $a \in \mathbb{Z}_\ell$, while $D(a)$ is in $\operatorname{SL}_2(\mathbb{Z}_\ell)$ if and only if $a \not \equiv -1 \pmod \ell$.

\section{Preliminary lemmas}
In this section we prove some general lemmas which will be used repeatedly throughout the paper.

\subsection{Logarithms and exponentials}
We recall the following fundamental properties of logarithms and exponentials, both for $\ell$-adic integers and for elements of $M_2(\mathbb{Z}_\ell)$, the set of $2 \times 2$ matrices with coefficients in $\mathbb{Z}_\ell$. The proofs of these statements are immediate upon direct inspection of the involved power series, and will not be included.

\begin{lemma}\label{lemma_log}
The power series
$
\displaystyle \log(1+t) = \sum_{n \geq 1} (-1)^{n-1} \frac{t^n}{n}
$
converges for $t \in \ell \mathbb{Z}_\ell$ (respectively for $t \in \ell M_2(\mathbb{Z}_\ell)$), and establishes a bijection between $1+\ell^{1+v}\mathbb{Z}_\ell$ and $\ell^{1+v}\mathbb{Z}_\ell$ (resp.~between $\mathcal{B}_\ell(1+v)$ and $\ell^{1+v}M_2(\mathbb{Z}_\ell)$). The inverse function is given by the power series
$
\displaystyle \exp t = \sum_{n \geq 0} \frac{t^n}{n!},
$
which converges for $t \in \ell^{1+v}\mathbb{Z}_\ell$ (resp.~$t \in \ell^{1+v}M_2(\mathbb{Z}_\ell)$).
Furthermore, for every $t \in \ell^{1+v}\mathbb{Z}_\ell$ we have $v_\ell \log(1+t)=v_\ell (t)$, and for every $t \in \ell^{1+v} \mathbb{Z}_\ell$ we have $v_\ell(\exp(t)-1)=v_\ell(t)$.
\end{lemma}

\begin{lemma}\label{lemma_CongruenceExponentials}
Let $A,B \in M_2(\mathbb{Z}_2)$ and $n \geq 2$ be an integer. Suppose that $A \equiv 0 \pmod 4$ and $B \equiv 0 \pmod {2^n}$: then $\exp(A+B) \equiv \exp(A) \pmod {2^n}$ and $\log(A+B) \equiv \log(A) \pmod {2^n}$.
\end{lemma}

\subsection{Subgroups of $\operatorname{SL}_2(\mathbb{Z}_\ell)^n$ and $\mathbb{Z}_\ell$-Lie algebras}
In this section we consider various properties of closed subgroups of $\operatorname{SL}_2(\mathbb{Z}_\ell)^n$, including their Lie algebras, generating sets, and derived subgroups.




\begin{lemma}\label{lemma_NIsWellDefined}
Let $n$ be a positive integer and $G$ be a closed subgroup of $\operatorname{SL}_2(\mathbb{Z}_\ell)^n$. The collection of all normal pro-$\ell$ subgroups of $G$ has a unique maximal element, which we denote by $N(G)$. Likewise, if $G$ is a subgroup of $\operatorname{SL}_2(\mathbb{F}_\ell)^n$, the collection of all normal subgroups of $G$ whose order is a power of $\ell$ admits a unique maximal element, which we denote again by $N(G)$.
\end{lemma}
\begin{proof}
Denote by $\pi:G \to \operatorname{SL}_2(\mathbb{F}_\ell)^n$ the canonical projection and let $N$ be a normal, pro-$\ell$ subgroup of $G$. Clearly $\pi(N)$ is an $\ell$-group that is normal in $G(\ell)$, so it suffices to show that $N(G(\ell))$ is well-defined, for then $N(G)=\pi^{-1} (N(G(\ell)))$ is the maximal normal pro-$\ell$ subgroup of $G$. To treat the finite case consider first $n=1$. Then $G$ is a subgroup of $\operatorname{SL}_2(\mathbb{F}_\ell)$, and it easy to see that the collection of its maximal normal subgroups of order a power of $\ell$ has a maximal element: it follows from the Dickson classification that this is given by the unique $\ell$-Sylow if $G$ is of Borel type, and by the trivial group otherwise. Finally, if $G$ is a subgroup of $\operatorname{SL}_2(\mathbb{F}_\ell)^n$ with $n>1$, we denote by $G_i$ the projection of $G$ on the $i$-th factor $\operatorname{SL}_2(\mathbb{F}_\ell)$; it is then immediate to check that
$
N(G) = \left\{ (g_1,\ldots,g_n) \in G \bigm\vert g_i \in N(G_i) \right\}.
$
\end{proof}

\begin{lemma}\label{lemma_ConjStableSubspaces}
Let $t$ be a non-negative integer. Let $W \subseteq \mathfrak{sl}_2(\mathbb{Z}_\ell)$ be a Lie subalgebra that does not reduce to zero modulo $\ell^{t+1}$. Suppose that $W$ is stable under conjugation by $\mathcal{B}_\ell(s)$ for some non-negative integer $s$, where $s \geq 2$ if $\ell=2$ and $s \geq 1$ if $\ell=3$ or $5$. Then $W$ contains the open set $\ell^{t+4s+4v} \mathfrak{sl}_2(\mathbb{Z}_\ell)$.
\end{lemma}

\begin{proof}
Fix an element $w$ of $W$ that does not vanish modulo $\ell^{t+1}$ and write $w=\mu_x x + \mu_y y + \mu_h h$ for some $\mu_x, \mu_y, \mu_h \in \mathbb{Z}_\ell$, where
\[
h=\left( \begin{matrix} 1 & 0 \\ 0 & -1 \end{matrix} \right), \; x=\left( \begin{matrix} 0 & 1 \\ 0 & 0 \end{matrix} \right), \; y=\left( \begin{matrix} 0 & 0 \\ 1 & 0 \end{matrix} \right)
\]
and $\min \left\{v_\ell(\mu_x),v_\ell(\mu_y), v_\ell(\mu_h) \right\} \leq t$. Set $\alpha:=1+\ell^s$ and consider the matrices
$C_\ell=L(\ell^s)$, $C_r=R(\ell^s),$ and $C_d=D(\ell^s)$.
For $\bullet \in \{l,d,r \}$ define $\mathcal{C}_\bullet$ as the linear operator of $W$ into itself given by
\[
\mathcal{C}_\bullet(x) = {C_\bullet}^{-1} x \, C_\bullet
\]
and set $\mathcal{D}_\bullet:=\mathcal{C}_\bullet - \operatorname{Id}$. As $W$ is stable under conjugation by $\mathcal{B}_\ell(s)$, it is in particular also stable under $\mathcal{C}_\bullet$ and $\mathcal{D}_\bullet$. One can easily check the following identities:
\[
\alpha^4 (\mathcal{C}_d- \alpha^2 \operatorname{Id}) \circ \mathcal{D}_d (w) = (\alpha^4-1)(\alpha^2-1) \mu_x x \in W\]
\[
(\alpha^2 \mathcal{C}_d- \operatorname{Id}) \circ \mathcal{D}_d (w) = (\alpha^4-1)(\alpha^2-1) \mu_y y \in W,
\]
and by considering the decomposition of $w$ we deduce that also $(\alpha^4-1)(\alpha^2-1) \mu_h h$ belongs to $W$. By the assumptions on $s$ we have $v_\ell((\alpha^4-1)(\alpha^2-1))=2s + 3v$, and therefore $W$ contains at least one of
\[
M_1=\ell^{t+2s+3v} h, \quad M_2=\ell^{t+2s+3v}x, \quad  M_3=\ell^{t+2s+3v} y.
\]
In the first case $W$ contains $\ell^{t+3s+4v}\mathfrak{sl}_2\left(\mathbb{Z}_\ell\right)$ because of the following identities:
\[
\mathcal{D}_r(M_1)=2 \ell^{t+3s+3v}x,  \quad \mathcal{D}_l(M_1)=-2 \ell^{t+3s+3v} y.
\]
In the second and third case $W$ contains $\ell^{t+4s+4v} \mathfrak{sl}_2\left(\mathbb{Z}_\ell\right)$ because of the following identities:
\[
\mathcal{D}_l \circ \mathcal{D}_l(M_2)=-2 \ell^{t+4s+3v} y, \quad \mathcal{D}_l \circ \mathcal{D}_l(M_2)-2\mathcal{D}_l(M_2)=-2 \ell^{t+3s+3v}h\]
\[
\mathcal{D}_r \circ \mathcal{D}_r(M_3)=-2 \ell^{t+4s+3v} x, \quad \mathcal{D}_r \circ \mathcal{D}_r(M_3)-2\mathcal{D}_r(M_3)=2 \ell^{t+3s+3v}h.\]
\end{proof}

\begin{lemma}\label{lemma_ApproximateEigenvalue1}
Let $n$ and $m$ be positive integers. Let $g \in \operatorname{End} \left(\mathbb{Z}_\ell^m\right)$ and write $p_g$ for the characteristic polynomial of $g$. Let furthermore $\lambda \in \mathbb{Z}_\ell, w \in \mathbb{Z}_\ell^m$ be such that $gw \equiv \lambda w \pmod{\ell^n}$. Suppose that $w \not\equiv 0 \pmod{\ell^{\alpha+1}}$ holds for some $0 \leq \alpha < n$. Then we have $p_g(\lambda) \equiv 0 \pmod{\ell^{n-\alpha}}$.
\end{lemma}

\begin{proof}
Denote by $(g-\lambda \operatorname{Id})^*$ the adjugate matrix of $(g-\lambda \operatorname{Id})$, that is the unique operator such that $(g-\lambda \operatorname{Id})^*(g-\lambda \operatorname{Id})=\det (g-\lambda \operatorname{Id}) \cdot \operatorname{Id}$ holds. Multiplying $(g-\lambda \operatorname{Id})w \equiv 0 \pmod{\ell^n}$ on the left by $(g-\lambda \operatorname{Id})^*$ we obtain $\det(g-\lambda \operatorname{Id}) \cdot  w \equiv 0 \pmod{\ell^n}$, and by considering the coordinate of $w$ of smallest valuation we obtain $p_g(\lambda)=\det(g-\lambda \operatorname{Id}) \equiv 0 \pmod{\ell^{n-\alpha}}$ as claimed.
\end{proof}

\begin{definition}
If $n \geq 2$ and if $g_1 , \ldots, g_n$ are elements of a group $G$, we write $\operatorname{Comm}_{n-1}(g_1, \ldots, g_n)$ for the 
$(n - 1)$ times iterated commutator, which is an element of $G$ that can be
defined via the following recursion:
\[
\operatorname{Comm}_{1}(g_1,g_2)=[g_1,g_2]=g_1g_2g_1^{-1}g_2^{-1},
\]
\[
\operatorname{Comm}_{n-1}(g_1 , \ldots, g_n ) = [\operatorname{Comm}_{n-2} (g_1, \ldots, g_{n-1}), g_n ].
\]
Furthermore, if $G_1,\ldots,G_n$ are subgroups of a topological group $G$, we write $\operatorname{Comm}_{n-1}\left(G_1,\ldots,G_n \right)$ for the subgroup of $G$ topologically generated by
\[
\left\{\operatorname{Comm}_{n-1}(g_1,\ldots,g_n) \bigm\vert g_i \in G_i \text{ for }i=1,\ldots,n\right\}.
\]
We also write $[G_1,G_2]$ for $\operatorname{Comm}_1(G_1,G_2)$.
\end{definition}

\begin{lemma}\label{lemma_Generators}Let $s$ be a non-negative integer (with $s \geq 2$ if $\ell=2$ and $s \geq 1$ if $\ell=3$). Let $a,b,c \in \mathbb{Z}_\ell$ all have exact valuation $s$. Suppose that $c \not \equiv -1 \pmod \ell$, so that $D(c)$ belongs to $\operatorname{SL}_2(\mathbb{Z}_\ell)$. Let $G$ be a closed subgroup of $\operatorname{SL}_2(\mathbb{Z}_\ell)$.
\begin{enumerate}
\item Suppose $G$ contains $L(a)$ (resp. $R(b)$): then for all $d \in \mathbb{Z}_\ell$ such that $v_\ell(d) \geq s$ the group $G$ contains $L(d)$ (resp. $R(d)$).
\item Suppose $s \geq 1$ and $G$ contains $D(c)$: then for all $d \in \mathbb{Z}_\ell$ such that $v_\ell(d) \geq s$ the group $G$ contains $D(d)$.
\item Suppose $\ell \geq 5$ or $s \geq 1$. If $G$ contains $L(a), R(b), D(c)$, then it contains all of $\mathcal{B}_\ell(s)$.
\item Suppose $\ell \geq 5$ and $G$ contains $L(a), R(b)$: then $G$ contains all of $\mathcal{B}_\ell(2s)$.


\end{enumerate}
\end{lemma}
\begin{proof}
\begin{enumerate}
\item Let $u=\frac{d}{a} \in \mathbb{Z}_\ell$. It is easy to check that for every integer $n$ we have $L(a)^n=L(an)$. Fix a sequence $n_k$ of rational integers that converge to $u$ in the $\ell$-adic topology: then we have
$
\lim_{k \to \infty} L(a)^{n_k} = \lim_{k \to \infty} L(an_k)=  L\left( \lim_{k \to \infty}  an_k\right) = L(au) = L(d).
$
Now observe that $G$ contains every term of the sequence $L(a)^{n_k}$, so -- since it is closed in $\operatorname{SL}_2(\mathbb{Z}_\ell)$ by assumption -- it also contains their limit $L(d)$. The same proof also applies to the case of $R(b)$.
\item The assumptions on $s$ imply that $\log(1+c)$ is defined, and that $v_\ell \log(1+c) = s$ (cf. lemma \ref{lemma_log}). Let $\displaystyle u=\frac{\log(1+d)}{\log(1+c)}$, which is a well-defined element of $\mathbb{Z}_\ell$ since $v_\ell \log(1+d) \geq s$. Choose as above a sequence $n_k$ that converges to $u$ in the $\ell$-adic topology: then $G$ contains the limit
\[
\begin{aligned}
\lim_{k \to \infty} D(c)^{n_k} & =  \lim_{k \to \infty} \left( \begin{matrix} (1+c)^{n_k} & 0 \\ 0 & (1+c)^{-n_k} \end{matrix} \right) \\
& = \lim_{k \to \infty} \left( \begin{matrix} \exp(n_k \log(1+c)) & 0 \\ 0 & \exp(-n_k \log(1+c)) \end{matrix} \right) \\
& = \left( \begin{matrix} \exp(u \log(1+c)) & 0 \\ 0 & \exp(-u \log(1+c)) \end{matrix} \right) = D(d).
\end{aligned}
\]

\item Suppose first $s=0$: then it is easy to see that $G(\ell)=\operatorname{SL}_2(\mathbb{F}_\ell)$. Since $\ell \geq 5$, by \cite[IV-23, Lemma 3]{SerreAbelianRepr} this implies $G=\operatorname{SL}_2(\mathbb{Z}_\ell)$. Suppose on the other hand that $s \geq 1$: then parts (1) and (2) imply that $G$ contains $L(d)$, $R(d)$, $D(d)$ for all $d \in \mathbb{Z}_\ell$ of valuation at least $s$, and by \cite[Lemma 3.1]{AdelicEC} these elements generate $\mathcal{B}_\ell(s)$.

\item If $s=0$ we notice -- as in part (3) -- that $G(\ell)=\operatorname{SL}_2(\mathbb{F}_\ell)$, hence $G=\operatorname{SL}_2(\mathbb{Z}_\ell)$. Otherwise, it suffices to apply part (3) to $(a^2,b^2,ab)$: indeed part (1) implies that $G$ contains $L(a^2), R(b^2)$ and
$
\displaystyle D(ab) = L\left( \frac{a}{1+ab} \right) R(-b) L(-a) R\left( \frac{b}{1+ab} \right).
$
\end{enumerate}
\end{proof}

\begin{lemma}\label{lemma_Commutator}
Let $s_1, \ldots, s_n$ be non-negative integers (where for every $j=1,\ldots,n$ we require that $s_j \geq 2$ if $\ell=2$ and $s_j \geq 1$ if $\ell=3$). The iterated commutator $\operatorname{Comm}_{n-1}(\mathcal{B}_\ell(s_1),\mathcal{B}_\ell(s_2), \ldots, \mathcal{B}_\ell(s_n))$ contains $B_\ell(s_1+\cdots+s_n+(n-1)v)$.
\end{lemma}

\begin{proof}
Consider first the case $n=2$.
Let $s$ be a non-negative integer (with $s \geq 2$ if $\ell=2$ and $s \geq 1$ if $\ell=3$). By lemma \ref{lemma_Generators} (3), if $a,b,c$ are any three elements of $\mathbb{Z}_\ell$ of valuation $s$ and such that $c \not \equiv -1 \pmod \ell$, the group $\mathcal{B}_\ell(s)$ is topologically generated by $L(a), R(b), D(c)$. It is easy to check the following identities:
\[
\left[ L(\ell^{s_1}), D(\ell^{s_2}) \right] = L\left( \frac{\ell^{s_2}+2}{\left(\ell^{s_2}+1\right)^2} \ell^{s_1+s_2} \right)
\]
\[
\left[ R(\ell^{s_1}), D(\ell^{s_2}) \right] = R\left( -(2+\ell^{s_2}) \ell^{s_1+s_2} \right),
\]
where (by the assumptions on $s_1,s_2$) we have
$\displaystyle
v_\ell\left(\frac{\ell^{s_2}+2}{\left(\ell^{s_2}+1\right)^2} \right) = v_\ell(2+\ell^{s_2})=v$. To conclude the proof for $n=2$ it thus suffices to show that $\left[\mathcal{B}_\ell(s_1), \mathcal{B}_\ell(s_2)  \right]$ contains an element of the form $D(c)$ with $v_\ell(c) \leq s_1+s_2+v$. This is easily achieved thanks to the following identity:
\[
L\left( - \frac{\ell^{s_1 + 2 s_2}}{1 + \ell^{s_1 + s_2} + \ell^{2 s_1 + 2 s_2}} \right) \left[ R(\ell^{s_1}), L(\ell^{s_2}) \right] R\left( \frac{\ell^{2 s_1+s_2}}{1+\ell^{s_1+s_2}+\ell^{2 s_1+2 s_2}} \right) = D\left(\ell^{s_1+s_2}+\ell^{2(s_1+s_2)}\right),
\]
where we know that $\displaystyle L\left( - \frac{\ell^{s_1 + 2 s_2}}{1 + \ell^{s_1 + s_2} + \ell^{2 s_1 + 2 s_2}} \right)$ and $\displaystyle  R\left( \frac{\ell^{2 s_1+s_2}}{1+\ell^{s_1+s_2}+\ell^{2 s_1+2 s_2}} \right)$ both belong to $\left[\mathcal{B}_\ell(s_1),\mathcal{B}_\ell(s_2) \right]$ by what we already proved and by lemma \ref{lemma_Generators} (1).
The case of an arbitrary $n$ follows by induction from the case $n=2$.
\end{proof}


The following lemma can be considered as an integral analogue of \cite[Lemma on p. 790]{MR0457455}.

\begin{lemma}\label{lemma_IntegralGoursat}
Let $n \geq 2$ be an integer and $G$ be a closed subgroup of $\prod_{i=1}^n \operatorname{SL}_2(\mathbb{Z}_\ell)$. Write $\pi_i$ for the projection on the $i$-th factor, and suppose that for every $i\neq j$ there is some non-negative integer $s_{ij}$ (with $s_{ij} \geq 2$ if $\ell=2$ and $s_{ij} \geq 1$ if $\ell=3$) such that the group $\left(\pi_i \times \pi_j\right)(G)$ contains $\mathcal{B}_\ell(s_{ij}, s_{ij})$. Then $G$ contains $\prod_{i=1}^n \mathcal{B}_\ell\left(\sum_{j \neq i} s_{ij}+(n-2)v\right)$.
\end{lemma}
\begin{proof}
Without loss of generality, by the symmetry of the problem it is enough to prove that $G$ contains
$
\left\{\operatorname{Id}\right\} \times \cdots \times \left\{\operatorname{Id}\right\} \times \mathcal{B}_\ell\left( \sum_{j \neq n} s_{nj} +(n-2)v \right).
$
By lemma \ref{lemma_Commutator} we know that the group $\mathcal{B}_\ell\left( \sum_{j \neq n} s_{nj} +(n-2)v \right)$ is generated by
$
\left\{ \operatorname{Comm}_{n-2}(g_1,\ldots,g_{n-1}) \bigm\vert g_j \in \mathcal{B}_\ell(s_{nj}) \right\},
$
so it suffices to show that for every choice of $(g_1,\ldots,g_{n-1}) \in \prod_{j \neq n} \mathcal{B}_\ell(s_{nj})$ the group $G$ contains the element $\left(\operatorname{Id}, \ldots,\operatorname{Id}, \operatorname{Comm}_{n-2}(g_1,\ldots,g_{n-1}) \right)$.
By hypothesis we can find $x_1, \ldots, x_{n-1} \in G$ such that $\pi_i(x_i)=\operatorname{Id}$ and $\pi_n(x_i)=g_i$ for all $i$ between $1$ and $n-1$. Consider now the iterated commutator
$
\tilde{g}=\operatorname{Comm}_{n-2}(x_1,\ldots, x_{n-1})\in G
$
. For $i \leq n-1$, the $i$-th component of $\tilde{g}$ is trivial because we have
$
\displaystyle \pi_i(\tilde{g})  = \operatorname{Comm}_{n-2}\left(\pi_i(x_1),\ldots,\underbrace{\pi_i(x_i)}_{\operatorname{Id}}, \ldots, \pi_i(x_{n-1}) \right)=\operatorname{Id};
$
moreover, our choice of $x_1, \ldots, x_n$ ensures that
$
\pi_n(\tilde{g})=\operatorname{Comm}_{n-2}\left( g_1,\ldots, g_{n-1}\right)
$ holds. We have thus shown that $(\operatorname{Id}, \ldots, \operatorname{Id}, \operatorname{Comm}_{n-2}\left( g_1,\ldots, g_{n-1}\right)
)$ is an element of $G$. 
\end{proof}

\begin{lemma}\label{lemma_TrivialModEllImpliesTrivial}
Let $\ell>2$ and $G$ be a closed subgroup of $\operatorname{SL}_2(\mathbb{Z}_\ell)^2$. Suppose $G$ contains an element $g=(x,y)$ with $[x]=[\pm \operatorname{Id}]$ and $[y]$  nontrivial and of order prime to $\ell$. Then $G$ contains an element of the form $(\pm \operatorname{Id},z)$, where the order of $[z]$ is the same as the order of $[y]$.
\end{lemma}

\begin{proof}
Such an element is given by any limit point of the sequence $g^{\ell^n}$.
\end{proof}

\begin{lemma}\label{lemma_Saturation}
Let $n \geq 1$ and $G_1, G_2$ be two closed subgroups of $\operatorname{GL}_2(\mathbb{Z}_\ell)^n$ (with $G_1(4), G_2(4)$ trivial if $\ell=2$). Suppose that the two groups $\left\{ \lambda g_1 \bigm\vert \lambda \in \mathbb{Z}_\ell^\times, g_1 \in G_1 \right\}$ and $\left\{ \mu g_2 \bigm\vert \mu \in \mathbb{Z}_\ell^\times, g_2 \in G_2 \right\}$ coincide: then $G_1,G_2$ have the same derived subgroup and the same Lie algebra.
 \end{lemma}
\begin{proof}
The hypothesis is symmetric in $G_1,G_2$, so it suffices to show the inclusions $G_1' \subseteq G_2'$ and $\mathcal{L}(G_1) \subseteq \mathcal{L}(G_2)$. The hypothesis implies in particular that for all $g_1 \in G_1$ there exists $\nu(g_1) \in \mathbb{Z}_\ell^\times$ such that $\nu(g_1) g_1 \in G_2$.
Now let $x,y$ be elements of $G_1$: then $[x,y]=[\nu(x)x,\nu(y)y] \in G_2'$, so -- since $G_1'$ is generated by elements of the form $[x,y]$ for $x,y$ varying in $G_1$ -- we have $G_1' \subseteq G_2'$. As for the Lie algebra, simply notice that for all $g_1 \in G_1$ we have $\Theta_n(g_1)=\frac{1}{\nu(g_1)}\Theta_n(\nu(g_1)g_1) \in \mathcal{L}(G_2)$.
\end{proof}

\begin{lemma}\label{lemma_NontrivialImpliesNontrivialAlgebra}
Let $n$ be a positive integer and $g \in \operatorname{SL}_2(\mathbb{Z}_\ell)$ a matrix satisfying $g \equiv \operatorname{Id} \pmod {\ell^{1+v}}$ and $g \not \equiv \operatorname{Id} \pmod{\ell^n}$. Then $\Theta_1(g)\not \equiv 0 \pmod{\ell^n}$.
\end{lemma}
\begin{proof}
Write $g=\left( \begin{matrix}a & b \\ c & d\end{matrix} \right)$, so that $\Theta_1(g)=\left(\begin{matrix} \frac{a-d}{2} & b \\ c & \frac{d-a}{2} \end{matrix} \right)$. Suppose by contradiction that we had $\Theta_1(g) \equiv 0 \pmod {\ell^n}$, that is $c \equiv d \equiv \frac{a-d}{2} \equiv 0 \pmod{\ell^n}$. It then follows $a \equiv d \pmod {\ell^{n+v}}$, and since $1=\det(g)=ad-bc$ we find $a^2 \equiv d^2 \equiv 1 \pmod{\ell^{n+v}}$. Since $a \equiv 1 \pmod {\ell^{1+v}}$ by hypothesis, this implies $a \equiv d \equiv 1 \pmod{\ell^n}$, that is, $g \equiv \operatorname{Id} \pmod{\ell^n}$, contradiction.
\end{proof}

\subsection{Teichm\"uller lifts}
We now recall the definition of Teichm\"uller lifts and some of their basic properties.

\begin{definition}
Let $F$ be a finite extension of $\mathbb{Q}_\ell$ with residue field $\mathbb{F}=\mathbb{F}_{\ell^k}$. For an element $[f] \in \mathbb{F}$ we denote by $\omega([f])$ the Teichm\"uller lift of $[f]$, that is to say the only element $g \in \mathcal{O}_F$ that reduces to $[f]$ in $\mathbb{F}$ and satisfies $g^{\ell^k}=g$.
\end{definition}

\begin{lemma}
With the notation of the previous definition, the sequence $f^{\ell^{kn}}$ converges to $\omega([f])$ when $n$ tends to infinity.
\end{lemma}

\begin{proof}
Write $f=\omega([f]) \cdot u$, where $u$ reduces to 1 in the residue field $\mathbb{F}$. It is clear that the sequence $u^{\ell^{kn}}$ converges to 1 in $F$, so $f^{\ell^{kn}} = \omega([f])^{\ell^{kn}} \cdot u^{\ell^{kn}} = \omega([f]) \cdot u^{\ell^{kn}}$ converges to $\omega([f])$.
\end{proof}

\begin{lemma}\label{lemma_Teichmuller1}
Let $g$ be an element of $\operatorname{SL}_2(\mathbb{Z}_\ell)$ such that $[g]$ has order prime to $\ell$ and strictly greater than 2: then the sequence $g^{\ell^{2n}}$ for $n \in \mathbb{N}$ converges to a certain $g_\infty$ that satisfies $g_\infty^{\ell^2}=g_\infty$.
Moreover, if $[g]$ is diagonalizable over $\mathbb{F}_\ell$, the limit $g_\infty$ even satisfies $g_\infty^\ell=g_\infty$.
\end{lemma}

\begin{proof}%
The assumption implies that $[g]$ has distinct eigenvalues, hence there exists an extension $F$ of $\mathbb{Q}_\ell$, of degree at most 2, over which $g$ can be written as $g=P^{-1} D P$, where $D=\operatorname{diag}(\lambda,\lambda^{-1})$ is diagonal and $P$ is a base-change matrix. By the previous lemma, the sequence $D^{\ell^{2n}}$ converges to $\operatorname{diag}(\omega([\lambda]),\omega([\lambda^{-1}]))$, so the sequence $g^{\ell^{2n}} = P D^{\ell^{2n}} P^{-1}$ converges to $P^{-1} \operatorname{diag}(\omega([\lambda]),\omega([\lambda^{-1}])) P$, which satisfies the conclusion since $\omega([\lambda])^{\ell^2}=\omega([\lambda])$. For the second statement we can take $F=\mathbb{Q}_\ell$ and use the fact that the Teichm\"uller lifts satisfy $\omega([\lambda])^\ell=\omega([\lambda])$.
\end{proof}


\begin{proof}
Let $[\lambda_1^{\pm 1}]=[\lambda_2^{\pm 1}]$ be the eigenvalues of $[g_1]=[g_2]$. As in the proof of the previous lemma we see that the sequence $(g_1,g_2)^{\ell^{2n}}$ converges to a limit $(h_1,h_2)$ such that the matrices $h_1,h_2$ have the same eigenvalues $\omega([\lambda_i]),\omega([\lambda_i^{-1}])$. The existence of $Q$ is then clear if $[h_1]=[h_2]$ is diagonalizable over $\mathbb{F}_\ell$, and follows from \cite[Lemma 4.7]{AdelicEC} otherwise. The claim on the orders of $[h_1], [h_2]$ is clear.
\end{proof}

\section{Odd $\ell$, $n=2$}\label{sect_Pinkl}
In this section we establish the case $n=2$ of the main theorem when $\ell$ is an odd prime, together with a refined version of it that applies to subgroups of $\operatorname{GL}_2(\mathbb{Z}_\ell)^2$:






\begin{theorem}\label{thm_PinkGL22}
Let $\ell>2$ be a prime number and $G$ an open subgroup of $\operatorname{GL}_2(\mathbb{Z}_\ell)^2$. Let $G_1, G_2$ be the projections of $G$ on the two factors $\operatorname{GL}_2(\mathbb{Z}_\ell)$, and for $i=1,2$ let $n_i$ be a positive integer such that $G_i$ contains $\mathcal{B}_\ell(n_i)$. Suppose furthermore that for every $(g_1,g_2) \in G$ we have $\det(g_1)=\det(g_2)$. At least one of the following holds:
\begin{enumerate}
\item $G$ contains $\mathcal{B}_\ell(20 \max\{n_1,n_2\},20 \max\{n_1,n_2\})$
\item there exists a subgroup $T$ of $G$, of index dividing $2 \cdot 48^2$, with the following properties:
\begin{enumerate}
\item if $\mathcal{L}(T)$ contains $\ell^k \mathfrak{sl}_2(\mathbb{Z}_\ell) \oplus \ell^k \mathfrak{sl}_2(\mathbb{Z}_\ell)$ for a certain integer $k \geq 0$, then $T$ contains $\mathcal{B}_\ell(p,p)$, where
$
p=2k+\max\left\{2k,8n_1,8n_2 \right\}.
$
\item for any $(t_1,t_2)$ in $T$, if both $[t_1]$ and $[t_2]$ are multiples of the identity, then they are equal.
\end{enumerate}

\end{enumerate}
\end{theorem}


We will derive this theorem from the corresponding statement for $\operatorname{SL}_2(\mathbb{Z}_2)^2$:

\begin{theorem}\label{thm_PinkSL22}
Let $\ell > 2$ be a prime number and $G$ an open subgroup of $\operatorname{SL}_2(\mathbb{Z}_\ell)^2$. For $i=1,2$ let $n_i$ be a positive integer such that $G_i$ contains $\mathcal{B}_\ell(n_i)$. At least one of the following holds:
\begin{enumerate}
\item $G'$ contains $\mathcal{B}_\ell(20\max\{n_1,n_2\},20\max\{n_1,n_2\})$
\item there exists a subgroup $T$ of $G$, of index dividing $48^2$, with the following properties:
\begin{enumerate} \item if $\mathcal{L}(T)$ contains $\ell^k \mathfrak{sl}_2(\mathbb{Z}_\ell) \oplus \ell^k \mathfrak{sl}_2(\mathbb{Z}_\ell)$ for a certain integer $k \geq 0$, then $T'$ contains $\mathcal{B}_\ell(p,p)$, where
$
p=2k+\max\left\{2k,8n_1,8n_2 \right\}.
$
\item for any $(t_1,t_2)$ in $T$, if both $[t_1]$ and $[t_2]$ are multiples of the identity, then they are equal.
\end{enumerate}
\end{enumerate}

\end{theorem}

\subsection{Strategy of proof}
Given the amount of technical details required to establish theorem \ref{thm_PinkSL22}, before plunging into the actual proof we try to give a general overview of the argument. Several ingredients are needed. To begin with, when an element of $G$ exists whose two projections $g_1,g_2$ on $G_1, G_2$ are sensibly different (for example, $g_1$ has trivial reduction modulo $\ell$ but $g_2$ does not), then this element can be used to show that $G$ is stable under certain `projections' from $G$ to $G_1 \times \{\operatorname{Id}\}$ and $\{\operatorname{Id}\} \times G_2$: in this case, the assumption that $G_1, G_2$ contain open topological balls is enough to conclude that the same is true for $G$ (in a quantitative manner). This line of argument is what leads to lemma \ref{lemma_NontrivialElementIsEasy}. Thus, in proving theorem \ref{thm_PinkSL22} one can assume -- among other things -- that $G(\ell)$ looks like the graph of an isomorphism: the reduction to this case is carried out in proposition \ref{prop_FirstReduction}.
Moreover, theorem \ref{thm_PinkSL22} is essentially known (and due to Pink) when $G$ is pro-$\ell$, cf.~theorem \ref{thm_Pink_GP} below. We would therefore like to reduce to the case of pro-$\ell$ groups, which we do by considering the maximal normal pro-$\ell$ subgroup of $G$ (denoted by $N(G)$ in all that follows). Most of the remaining part of the proof (sections \ref{sect_FirstCase} to \ref{sect_LastCase}) is concerned with showing that when the Lie algebra of $T$ (a certain subgroup of $G$ of bounded index) is large, the same is true for the Lie algebra of $N(T)$, to which we can subsequently apply Pink's theorem. The key observation in carrying out this last step is the following: since $N(T)$ is (by definition) normal in $T$, its Lie algebra is acted upon by all of $T$, including those elements that do not belong to $N(T)$. A careful study of the action of these elements, whose details depend on the specific shape of $T(\ell)$, shows that the stability of $\mathcal{L}(N(T))$ by the conjugation action of $T$ forces it to be not very different from $\mathcal{L}(T)$, and an application of Pink's theorem (in the form of lemma \ref{lemma_NontrivialLieAlgebra}) then concludes the argument.

One final technical ingredient is our ability to change bases to simplify the computations. More precisely, let $g$ be any element of $ \operatorname{GL}_2(\mathbb{Z}_\ell)^2$: then the group $G$ satisfies the assumptions of theorem \ref{thm_PinkSL22} (for certain values of $n_1,n_2$) if and only if $gGg^{-1}$ does, because the topological balls $\mathcal{B}_\ell(s)$ are invariant under any integral change of basis. Likewise, the conclusion of the theorem only involves topological balls, multiples of the identity, and open sets of the form $\ell^k \mathfrak{sl}_2(\mathbb{Z}_\ell) \oplus \ell^k \mathfrak{sl}_2(\mathbb{Z}_\ell)$: all of these objects are invariant under change of basis -- notice that $\Theta(g t g^{-1})=g \Theta(t) g^{-1}$ for any $t \in \operatorname{GL}_2(\mathbb{Z}_\ell)^2$ -- and therefore the claim of theorem \ref{thm_PinkSL22} is true for $G$ if and only if it is true for $gGg^{-1}$. In other words, as it should be expected, theorem \ref{thm_PinkSL22} depends on $G$ only through its conjugacy class, and therefore we are free to change basis every time doing so simplifies some of our calculations.

\begin{remark}
The proof of theorem \ref{thm_PinkSL22} is constructive, in the sense that the group $T$ -- when we are in case (2) -- is described explicitly. Property (2b) will be clear from the construction, and we will not comment further on it.
\end{remark}

\subsection{Preliminary reductions}
One of the key ingredients of the proof of theorem \ref{thm_PinkSL22} is the following result:


\begin{theorem}\label{thm_Pink_GP}
Let $\ell > 2$ be a prime number and $k$ be a non-negative integer.
Suppose $G$ is a closed pro-$\ell$ subgroup of $\operatorname{SL}_2(\mathbb{Z}_\ell)$ such that $\mathcal{L}(G)$ contains $\ell^k \mathfrak{sl}_2(\mathbb{Z}_\ell)$: then $G'$ contains $\mathcal{B}_\ell(2k)$. 
Similarly, if $G$ is a closed pro-$\ell$ subgroup of $\operatorname{SL}_2(\mathbb{Z}_\ell)^2$ and $\mathcal{L}(G)$ contains $\ell^k \mathfrak{sl}_2(\mathbb{Z}_\ell) \oplus \ell^k \mathfrak{sl}_2(\mathbb{Z}_\ell)$, then $G'$ contains $\mathcal{B}_\ell(2k,2k)$. 
\end{theorem}

\begin{proof} The proof is identical in the two cases, so let us only consider subgroups of $\operatorname{SL}_2(\mathbb{Z}_\ell)^2$ (the case of $\operatorname{SL}_2(\mathbb{Z}_\ell)$ is also treated in \cite[proof of theorem 4.2]{AdelicEC}). Consider the basis
\[
x=\left(\begin{matrix} 0 & 1 \\ 0 & 0 \end{matrix} \right), \, y=\left(\begin{matrix} 0 & 0 \\ 1 & 0 \end{matrix} \right), \, h=\left(\begin{matrix} 1 & 0 \\ 0 & -1 \end{matrix} \right)
\]
of $\mathfrak{sl}_2(\mathbb{Z}_\ell)$. Since $[h,x]=2x$, $[h,y]=-2y$ and $[x,y]=h$, the derived subalgebra of $\ell^k\mathfrak{sl}_2(\mathbb{Z}_\ell)$ is $\ell^{2k}\mathfrak{sl}_2(\mathbb{Z}_\ell)$.
By assumption $\mathcal{L}(G)$ contains $\ell^k \mathfrak{sl}_2(\mathbb{Z}_\ell) \oplus \ell^k \mathfrak{sl}_2(\mathbb{Z}_\ell)$, so the derived subalgebra $[\mathcal{L}(G),\mathcal{L}(G)]$ contains $\ell^{2k}\mathfrak{sl}_2(\mathbb{Z}_\ell) \oplus \ell^{2k}\mathfrak{sl}_2(\mathbb{Z}_\ell)$. Furthermore, in the notation of \cite{MR1241950} we have that
$
C=\operatorname{tr}(\mathcal{L}(G) \cdot \mathcal{L}(G))$ contains both 
$\operatorname{tr} \left((\ell^k h,0) \cdot (\ell^kh,0) \right) = (2\ell^{2k}, 0)$ and $(0,2\ell^{2k})$. Since $C$ is by definition a topologically closed additive subgroup of $\mathbb{Z}_\ell^2$, this shows in particular that $C$ contains $\ell^{2k}\mathbb{Z}_\ell \oplus \ell^{2k}\mathbb{Z}_\ell$.
It then follows from \cite[Theorem 2.7]{MR1241950} that
\[
G'  = \left\{ g \in \operatorname{SL}_2(\mathbb{Z}_\ell)^2 \bigm\vert \Theta_2(g) \in [\mathcal{L}(G),\mathcal{L}(G)], \; \operatorname{tr}(g) - 2 \in C \right\} \supseteq  \mathcal{B}_\ell(2k,2k).
\]
\end{proof}

\begin{proof}{(of theorem \ref{thm_PinkGL22} assuming theorem \ref{thm_PinkSL22})}
Write $\det^*$ for the map
$
G \xrightarrow{\det} \mathbb{Z}_\ell^2 \xrightarrow{\pi_1} \mathbb{Z}_\ell$
given by the composition of the usual determinant with the projection on the first coordinate of $\mathbb{Z}_\ell^2$. By assumption, an element $(g_1,g_2)$ of $G$ satisfies $\det(g_1,g_2)=(1,1)$ if and only if it satisfies $\det^*(g_1,g_2)=1$.

Assume first that $\ell \leq 5$. Denote by $\tilde{T}$ the inverse image in $G$ of an $\ell$-Sylow of $G(\ell)$, and set
\[
T:=\ker\left( \tilde{T} \xrightarrow{\det^*} \mathbb{Z}_\ell^\times \to \frac{\mathbb{Z}_\ell^\times}{\mathbb{Z}_\ell^{\times 2}} \right).
\]
Using the fact $\det g_1=\det g_2$ for every $(g_1,g_2)\in G$, and that the index of an $\ell$-Sylow of $G(\ell)$ is prime to $\ell$, it is easily checked that the index of $T$ in $G$ divides $2 \cdot \frac{1}{\ell-1} \left( \frac{|\operatorname{GL}_2(\mathbb{F}_\ell)|}{\ell} \right)^2 \bigm\vert 2 \cdot 48^2$. Now set
$
T^{1} = \left\{ \lambda g \bigm\vert \lambda \in \mathbb{Z}_\ell^\times, g \in T \right\} \cap \operatorname{SL}_2(\mathbb{Z}_\ell)^2.
$ 
Using lemma \ref{lemma_Saturation} and the argument of \cite[Lemma 3.16 (2)]{AdelicEC} one sees easily that $T$ and $T^1$ have the same derived subgroup and the same Lie algebra, and moreover it is clear by construction that $T^1$ is a pro-$\ell$ subgroup of $\operatorname{SL}_2(\mathbb{Z}_\ell)^2$. Furthermore, every element $(t_1,t_2)$ of $T$ reduces to $([\operatorname{Id}],[\operatorname{Id}])$ modulo $\ell$, so $T$ satisfies (2b). Now if $\mathcal{L}(T)$ contains $\ell^k \mathfrak{sl}_2(\mathbb{Z}_\ell) \oplus \ell^k \mathfrak{sl}_2(\mathbb{Z}_\ell)$, then the same is true for $\mathcal{L}\left(T^1\right)$, hence $(T^1)'=T'$ contains $\mathcal{B}_\ell(2k,2k)$ by theorem \ref{thm_Pink_GP}, and $T$ has properties (2a) and (2b) as required.


Next consider the case $\ell>5$.
Let $U_1$ be the subgroup of $G$, of index at most 2, given by $\operatorname{ker}\left(G \xrightarrow{\det^*} \mathbb{Z}_\ell^\times \to \mathbb{Z}_\ell^\times/\mathbb{Z}_\ell^{\times 2} \right)$. 
Let $U_2 = \left\{ \lambda g \bigm\vert \lambda \in \mathbb{Z}_\ell^\times, g \in U_1 \right\} \cap \operatorname{SL}_2(\mathbb{Z}_\ell)^2$, and notice as above that $U_1$ and $U_2$ have the same derived subgroup. Also notice that $U_2$ is open in $\operatorname{SL}_2(\mathbb{Z}_\ell)^2$: indeed, as $G$ is open in $\operatorname{GL}_2(\mathbb{Z}_\ell)^2$ there is an $r>0$ such that $\mathcal{B}_\ell(r,r) \subseteq G$, and the definition of $U_2$ shows that $\mathcal{B}_\ell(r,r)$ is also contained in $U_2$. Moreover, since elements in $\mathcal{B}_\ell(n_1), \mathcal{B}_\ell(n_2)$ have unit determinant, the two projections of $U_2$ on the factors $\operatorname{GL}_2(\mathbb{Z}_\ell)$ contain $\mathcal{B}_\ell(n_1), \mathcal{B}_\ell(n_2)$ respectively.
We can assume that $U_2'=U_1'$ does not contain $\mathcal{B}_\ell(20\max\{n_1,n_2\},20\max\{n_1,n_2\})$, for otherwise the same holds for $G$ and we are done. Apply then theorem \ref{thm_PinkSL22} to $(U_2,n_1,n_2)$ to find a subgroup $T_2$ of $U_2$ (of index dividing $48^2$) that has properties (2a) and (2b) of that statement.
Notice that we have a well-defined morphism
$
U_1 \stackrel{\psi}{\rightarrow} U_2/(\pm (\operatorname{Id},\operatorname{Id}))
$
sending $g$ to the class of $g/\sqrt{\det^*(g)}$, where $\sqrt{\det^*{g}}$ exists in $\mathbb{Z}_\ell^\times$ by construction of $U_1$. Let $\overline{T_2}$ be the image of $T_2$ in the quotient $U_2/(\pm (\operatorname{Id},\operatorname{Id}))$. Now $\psi$ is surjective by definition of $U_2$, so if we define $T$ to be the inverse image of $\overline{T_2}$ through $\psi$, then the index $[U_1:T]$ divides $48^2$ and $[G:T]$ divides $2 \cdot 48^2$. 
Furthermore, the Lie algebra of $T$ and that of $T_2$ agree, as do their derived subgroups (lemma \ref{lemma_Saturation}).
Suppose now that $\mathcal{L}(T)$ contains $\ell^k \mathfrak{sl}_2(\mathbb{Z}_\ell) \oplus \ell^k \mathfrak{sl}_2(\mathbb{Z}_\ell)$: then the same is true for $\mathcal{L}(T_2)$, and therefore by property (2a) of theorem \ref{thm_PinkSL22} we see that $T_2'=T'$ contains $\mathcal{B}_\ell(p,p)$. 
Finally, let $(t_1,t_2)$ be in $T$ and suppose that $[t_1],[t_2]$ are multiples of the identity. By construction, there exists a scalar $\lambda \in \mathbb{Z}_\ell^\times$ and an element $(w_1,w_2) \in T_2$ such that $(t_1,t_2)=\lambda(w_1,w_2)$; in particular, $[w_1], [w_2]$ are multiples of the identity, so the properties of $T_2$ force $[t_1]=[\lambda][w_2]=[\lambda][w_1]=[t_2]$.
\end{proof}

The following proposition allows us to assume that $\ell > 5$:

\begin{proposition}\label{prop_PinkTrueSmallEll}
Theorem \ref{thm_PinkSL22} is true for $\ell \leq 5$.
\end{proposition}
\begin{proof} Take $T$ to be the inverse image in $G$ of a $\ell$-Sylow of $G(\ell)$. It is clear that $T$ satisfies (2b). Furthermore, $T$ is pro-$\ell$, so an application of theorem \ref{thm_Pink_GP} shows that $T$ satisfies (2a); finally, it is immediate to check that $[G:T]$ divides $\left(\frac{|\operatorname{SL}_2(\mathbb{F}_\ell)|}{\ell}\right)^2 \bigm\vert 48^2$.
\end{proof}

\begin{assumption}\label{ass_EllGreater5}
From now on, we work under the additional assumption that $\ell > 5$.
\end{assumption}
Our final objective is to compute, in terms of $k, n_1$ and $n_2$, an integer $p$ such that $T$ contains $\mathcal{B}_\ell(p,p)$. This would be immediate if $T$ were a pro-$\ell$ group, for then we would simply apply theorem \ref{thm_Pink_GP} as it is. In general, however, one needs to take into account the structure of $T(\ell)$, and many different possibilities arise, according to the type of $T_1(\ell),T_2(\ell)$ in the Dickson classification.
In some situations which we now discuss, the two projections $G_1$ and $G_2$ behave essentially independently of one another; in this case the problem is greatly simplified, and it is possible to exhibit an integer $p$ as above without examining too closely the structure of $G(\ell)$:
\begin{lemma}\label{lemma_NontrivialElementIsEasy}
Let $\ell>5$ and $G$, $n_1, n_2$ be as in theorem \ref{thm_PinkSL22}.
Suppose that $G$ contains an element $(a,b)$ such that $[a]=[\pm \operatorname{Id}]$ and the prime-to-$\ell$ part of the order of $[b]$ is at least 3. Then $G'$ contains $\mathcal{B}_\ell(4n_1+16n_2,8n_2) \supset \mathcal{B}_\ell(20\max\{n_1,n_2\},20\max\{n_1,n_2\})$. The same conclusion is true if $[b]=[\pm \operatorname{Id}]$ and the prime-to-$\ell$ part of the order of $[a]$ is at least 3.
\end{lemma}

\begin{proof}
The statement is clearly symmetric in $a$ and $b$, so let us assume $[a]=[\pm \operatorname{Id}]$.
Since $\ell^2$ does not divide the order of $\operatorname{SL}_2(\mathbb{F}_\ell)$, the element $[b]^{\ell}$ has order prime to $\ell$; replacing $(a,b)$ with $(a,b)^{\ell}$ allows us to assume that the order of $[b]$ is prime to $\ell$ and not less than 3. By lemma \ref{lemma_TrivialModEllImpliesTrivial}, $G$ contains an element of the form $(\pm \operatorname{Id},b')$, where the order of $[b']$ is the same as the order of $[b]$. We can therefore assume $a=\pm \operatorname{Id}$.
By hypothesis, for any $g_2$ in $\mathcal{B}_\ell(n_2)$ there exists a $g_1 \in G_1$ such that $(g_1,g_2)$ belongs to $G$: it follows that for any $g_2 \in \mathcal{B}_\ell(n_2)$ the element \begin{equation}\label{eq_CommutatorbPrime}
(\pm \operatorname{Id},b')^{-1}(g_1,g_2)(\pm \operatorname{Id},b')(g_1,g_2)^{-1}=(\operatorname{Id},(b')^{-1}g_2b'g_2^{-1})=(\operatorname{Id},[(b')^{-1},g_2])
\end{equation}
belongs to $G$. Up to a choice of basis, we can assume that either $b' = \left( \begin{matrix} d & 0 \\ 0 & 1/d \end{matrix} \right)$ for a certain unit $d$, or $b' = \left( \begin{matrix} c & d \varepsilon \\ d & c \end{matrix} \right)$ for certain units $c,d$ and a certain $\varepsilon$ such that $[\varepsilon]$ is not a square, or $b' = \left( \begin{matrix} 0 & -1 \\ 1 & 0 \end{matrix} \right)$ (for the last two cases cf.~\cite[Lemma 4.7]{AdelicEC}). In the first case, setting $g_2=R(\ell^{n_2})$ in \eqref{eq_CommutatorbPrime} shows that
$
\left( \operatorname{Id}, R\left(\left(d^{-2}-1\right) \ell^{n_2}\right) \right)
$
belongs to $G$. Given that $d$ is not congruent to $\pm 1$ modulo $\ell$, for otherwise the order of $[b]$ would be 1 or 2, we see that the $\ell$-adic valuation of $(d^{-2}-1) \ell^{n_2}$ is exactly $n_2$. Similarly, $G$ contains $\left( \operatorname{Id}, L\left( (d^2-1) \ell^{n_2} \right) \right)$, and by lemma \ref{lemma_Generators} (4) this implies that $G$ contains $\left\{1\right\} \times \mathcal{B}_\ell(2n_2)$. In the second case we notice that $R\left( -\frac{2 c \ell^{2 n_2} \left(2+\ell^{n_2}\right)^2}{d \left(1+\ell^{n_2}\right)^4}\right)$ can be written as
\[
\left[ (b')^{-1}, \left(
\begin{array}{cc}
 1+\ell^{n_2} & -\frac{c \ell^{n_2} (2+\ell^{n_2})}{d (1+\ell^{n_2})} \\
 0 & \frac{1}{1+\ell^{n_2}} \\
\end{array}
\right) \right] \cdot \left[ (b')^{-1}, \left(
\begin{array}{cc}
 \frac{1}{1+\ell^{n_2}} & \frac{c \ell^{n_2} \left(2+\ell^{n_2}\right)}{d \left(1+\ell^{n_2}\right)} \\
 0 & 1+\ell^{n_2} \\
\end{array}
\right) \right],
\]
hence $\left(\operatorname{Id},R\left(-\frac{2 c \ell^{2 n_2} \left(2+\ell^{n_2}\right)^2}{d \left(1+\ell^{n_2}\right)^4}\right)\right)$ belongs to $G$; by lemma \ref{lemma_Generators} (1) we then see that $G$ contains $\left(\operatorname{Id}, R\left(\ell^{2n_2}\right)\right)$, and analogous identities prove that $G$ also contains $\left(\operatorname{Id}, L\left(\ell^{2n_2}\right)\right)$. Lemma \ref{lemma_Generators} (4) then implies that $G$ contains $\{\operatorname{Id}\} \times \mathcal{B}_\ell(4n_2)$.
Finally, in the last case we use the identities
\[
\left[(b')^{-1}, D(\ell^{n_2}) \right]^{-1}=D(\ell^{n_2}(2+\ell^{n_2}))
\]
\[
\left[ D(\ell^{n_2}),R(\ell^{n_2}) \right]=R\left(\ell^{2n_2} (2+\ell^{n_2})  \right), \quad \left[ D(\ell^{n_2}),L(\ell^{n_2}) \right]=L\left(-\ell^{2n_2} \frac{2+\ell^{n_2}}{(1+\ell^{n_2})^2}  \right)
\]
which prove that $G$ contains $\{\operatorname{Id}\} \times \mathcal{B}_\ell(2n_2)$.

Consider now an element $h=(h_1,h_2)$ of $G$ whose first coordinate is $h_1=R\left( \frac{1}{\ell^2-1} \ell^{n_1}\right)$; such an element exists by hypothesis. The $\ell(\ell^2-1)$-th power of $h$ is of the form $h'=\left( R \left( \ell^{n_1+1} \right), h_2' \right)$, where $[h_2']=[h_2]^{\ell(\ell^2-1)}=[h_2]^{\left|\operatorname{SL}_2(\mathbb{F}_\ell)\right|}=[\operatorname{Id}]$. The $\ell^{4n_2-1}$-th power of $h'$ (recall that $n_2>0$), therefore, is a certain
$
h'' = \left( R\left( \ell^{n_1+4n_2} \right), h_2'' \right),
$
where $h_2'' \in \mathcal{B}_\ell(4n_2)$. By what we already proved, $G$ contains $(\operatorname{Id},h_2'')$, so it also contains
$
h'' \cdot (\operatorname{Id},h_2'')^{-1}=\left( R\left( \ell^{n_1+4n_2}\right), \operatorname{Id} \right).
$
The same argument shows that $\left( L\left( \ell^{n_1+4n_2}\right), \operatorname{Id} \right)$ belongs to $G$, and we finally deduce that $G$ contains $\mathcal{B}_\ell(2n_1+8n_2) \times \left\{\operatorname{Id}\right\}$, hence -- since we also have $G \supseteq \left\{\operatorname{Id}\right\} \times \mathcal{B}_\ell(4n_2)$ -- that $G$ contains $\mathcal{B}_\ell(2n_1+8n_2,4n_2)$. Taking derived subgroups and using lemma \ref{lemma_Commutator} then shows that
 $G'$ contains $\mathcal{B}_\ell(4n_1+16n_2,8n_2)$.
\end{proof}

\begin{lemma}\label{lemma_NontrivialLieAlgebra}
Let $\ell > 5$ be a prime, $G$ a closed subgroup of $\operatorname{SL}_2(\mathbb{Z}_\ell)^2$ and for $i=1,2$ let $n_i$ be a \textit{non-negative} integer such that $G_i$ contains $\mathcal{B}_\ell(n_i)$. Let $t$ be a non-negative integer and $u$ be an element of $\mathfrak{sl}_2(\mathbb{Z}_\ell)$ such that $u \not \equiv 0 \pmod{\ell^{t+1}}$.
Suppose that $\mathcal{L}(N(G)) \subseteq \mathfrak{sl}_2(\mathbb{Z}_\ell)\oplus \mathfrak{sl}_2(\mathbb{Z}_\ell)$ contains $\left( 0, u \right)$: then $G'$ contains $\left\{\operatorname{Id}\right\} \times \mathcal{B}_\ell(2t+8n_2)$.
Likewise, if $u_1, u_2 \in \mathfrak{sl}_2(\mathbb{Z}_\ell)$ are such that $u_1, u_2 \not \equiv 0 \pmod{\ell^{t+1}}$ and $\mathcal{L}(N(G))$ contains $(u_1,0)$ and $(0,u_2)$, then $G'$ contains $\mathcal{B}_\ell(2t+8n_1, 2t+8n_2)$.
\end{lemma}

\begin{proof}
Note first that -- by the same argument as \cite[Lemma 4.5]{AdelicEC} -- the Lie algebra $\mathcal{L}(N(G))$ is stable under conjugation by $G$. The smallest Lie subalgebra of $\mathcal{L}(N(G))$ that contains $(0,u)$ and is stable under conjugation by $G$ is $0 \oplus S(u)$, where $S(u)$ is the smallest Lie subalgebra of $\mathfrak{sl}_2(\mathbb{Z}_\ell)$ that contains $u$ and is stable under conjugation by $G_2$. By virtue of lemma \ref{lemma_ConjStableSubspaces}, and given that $G_2$ contains $\mathcal{B}_\ell(n_2)$, the algebra $S(u)$ contains $\ell^{t+4n_2} \mathfrak{sl}_2(\mathbb{Z}_\ell)$. It follows that $\mathcal{L}(N(G))$ contains $0 \oplus \ell^{t+4n_2} \mathfrak{sl}_2(\mathbb{Z}_\ell)$, and applying theorem \ref{thm_Pink_GP} we deduce that $N(G)'$ (hence also $G'$) contains $\left\{\operatorname{Id}\right\} \times \mathcal{B}_\ell(2t+8n_2)$. The second statement is now immediate.
\end{proof}


We thus see that for the following three categories of groups information on $G$ can be extracted from information on $\mathcal{L}(G)$:
\begin{itemize}
\item[(A)] pro-$\ell$ groups: by theorem \ref{thm_Pink_GP}, if $\mathcal{L}(G)$ contains $\ell^k\mathfrak{sl}_2(\mathbb{Z}_\ell) \oplus \ell^k\mathfrak{sl}_2(\mathbb{Z}_\ell)$, then $G'$ contains $\mathcal{B}_\ell(2k,2k)$;
\item[(B)] groups that contain an element $(a,b)$ such that $[a]$ is trivial and the prime-to-$\ell$ part of the order of $[b]$ is least 3, because of lemma \ref{lemma_NontrivialElementIsEasy};
\item[(C)] groups satisfying the hypotheses of lemma \ref{lemma_NontrivialLieAlgebra}.
\end{itemize}


We now start with a general open subgroup $G$ of $\operatorname{SL}_2(\mathbb{Z}_\ell)^2$ and show that (up to passing to a subgroup of finite, absolutely bounded index) the group $G$ must satisfy one of these three sets of hypotheses.
As already anticipated, we will need a case analysis based on the structure of $G_1(\ell)$ and $G_2(\ell)$. These are subgroups of $\operatorname{SL}_2(\mathbb{F}_\ell)$, and by the Dickson classification we know that any subgroup of $\operatorname{SL}_2(\mathbb{F}_\ell)$ is of one of the following types:
split Cartan, nonsplit Cartan, normalizer of split Cartan, normalizer of nonsplit Cartan, Borel, exceptional, all of $\operatorname{SL}_2(\mathbb{F}_\ell)$.

\begin{remark}
To be more precise we should rather write `contained in a split Cartan subgroup', `contained in a Borel subgroup', etc. The slight abuse of language should not cause any confusion.
\end{remark}

We call `type' of $G$ the pair $(\text{type of }G_1(\ell), \text{type of }G_2(\ell))$; the proof of theorem \ref{thm_PinkSL22} will proceed by analysing all the possibilities for the type of $G$. The following proposition helps cut down the total number of cases we need to treat. 

\begin{proposition}\label{prop_FirstReduction}
Let $\ell >5$ and $G, n_1, n_2$ be as in theorem \ref{thm_PinkSL22}. At least one of the following holds:
\begin{enumerate}
\item $G'$ contains $\mathcal{B}_\ell(20 \max\{n_1,n_2\},20 \max\{n_1,n_2\})$
\item there exists a pro-$\ell$ subgroup $T$ of $G$ such that the index $[G:T]$ divides $48^2$, and such that if $(t_1,t_2)$ is an element of $T$ for which $[t_1], [t_2]$ are both multiples of the identity, then $[t_1]=[t_2]$
\item there exists a subgroup $T$ of $G$ such that the index $[G:T]$ divides $4 \cdot 48$, and which enjoys all of the following properties (we denote by $T_1, T_2$ the projections of $T$ on the two factors $\operatorname{SL}_2(\mathbb{Z}_\ell)$):
\begin{enumerate} \item the projection of $T(\ell)$ in $T_1(\ell)/N(T_1(\ell)) \times T_2(\ell)/N(T_2(\ell))$ is the graph of an isomorphism $T_1(\ell)/N(T_1(\ell)) \cong T_2(\ell)/N(T_2(\ell))$
\item the type of $T_1(\ell)$ (resp. $T_2(\ell)$) is either Borel, split Cartan, nonsplit Cartan, or $\operatorname{SL}_2(\mathbb{F}_\ell)$
\item the orders of $T_1(\ell)/N(T_1(\ell))$ and $T_2(\ell)/N(T_2(\ell))$ do not divide 8
\item $T_1$ (resp. $T_2$) contains $\mathcal{B}_\ell(n_1)$ (resp. $\mathcal{B}_\ell(n_2)$)
\item if $(t_1,t_2) \in T$ is such that $[t_1]$ and $[t_2]$ are both multiples of the identity, then $[t_1]=[t_2]$
\end{enumerate}
\end{enumerate}
\end{proposition}
\begin{proof}
If (1) holds we are done, so we can assume this is not the case. The conclusion of lemma \ref{lemma_NontrivialElementIsEasy} is then false, and therefore so is its hypothesis: if $(a,b)$ is an element of $G$ with $[a]$ (resp. $[b]$) equal to $[\pm \operatorname{Id}]$, then the prime-to-$\ell$ part of the order of $[b]$ (resp. $[a]$) is at most 2.
Consider the kernel $J_2$ of the reduction $G(\ell) \to G_1(\ell)$, which we identify to a (normal) subgroup of $G_2(\ell)$. We claim that the prime-to-$\ell$ part of the order of $J_2$ is at most 2. Indeed, $\operatorname{SL}_2(\mathbb{F}_\ell)$ contains only one element of order 2, namely minus the identity, so if the prime-to-$\ell$ part of $|J_2|$ is at least 3, then $J_2$ contains an element $[b]$ whose order has prime-to-$\ell$ part at least 3, contradiction.
Taking into account the fact that $J_2$ is a subgroup of $\operatorname{SL}_2(\mathbb{F}_\ell)$, we see that the $\ell$-part of its order can only be 1 or $\ell$. Thus the order of $J_2$ can only be one of $1,2,\ell,2\ell$; furthermore, the last two cases can only happen if $G_2(\ell)$ is of Borel type, since these are the only subgroups of $\operatorname{SL}_2(\mathbb{F}_\ell)$ admitting a normal $\ell$-Sylow. By Goursat's lemma, we know that the projection of $G(\ell)$ in $G_1(\ell)/J_1 \times G_2(\ell)/J_2$ is the graph of an isomorphism $G_1(\ell)/J_1 \cong G_2(\ell)/J_2$.
Notice now that if $T$ is any subgroup of $G$ of index dividing $4 \cdot 48$, then property (3d) is automatically true for $T$: indeed, since $[G_1:T_1]$ divides $[G:T]$, which in turn divides $4 \cdot 48$, the index $[G_1:T_1]$ is prime to $\ell$, and therefore $\mathcal{B}_\ell(n_1)$ (which is a pro-$\ell$ group, since $n_1 \geq 1$) is contained in $G_1$ if and only if it is contained in $T_1$; the same argument applies to $T_2$ and $\mathcal{B}_\ell(n_2)$. Furthermore, it is clear that property (3e) follows from property (3a): indeed, the only multiples of the identity in $\operatorname{SL}_2(\mathbb{F}_\ell)$ are $\pm \operatorname{Id}$, and the existence in $T(\ell)$ of an element of the form $([\pm \operatorname{Id}],[\mp \operatorname{Id}])$ would contradict (3a). 

We now construct a group $T$ with the desired properties in two successive steps. We claim first that there is a subgroup $H$ of $G$, of index dividing 24, such that $H_1(\ell)$ and $H_2(\ell)$ are either (split or nonsplit) Cartan, Borel, or all of $\operatorname{SL}_2(\mathbb{F}_\ell)$, so that $H$ satisfies property (3b). This is easily seen by a case-by-case analysis depending on the type of $G_1(\ell)$:
\begin{itemize}
\item split/nonsplit Cartan: then $G_1(\ell)/J_1$ is cyclic, so the same is true for $G_2(\ell)/J_2$. It is easily checked that this is only possible if $G_2(\ell)$ is either Borel or Cartan, and we are done (taking $H=G$).
\item normalizer of split/nonsplit Cartan: there is a subgroup $H$ of $G$ of index at most 2 such that $H_1(\ell)$ is Cartan, and we are back to the previous case.
\item Borel: if the order of $J_1$ is prime to $\ell$, then the isomorphism $G_1(\ell)/J_1 \cong G_2(\ell)/J_2$ forces $G_2(\ell)$ to be Borel (because then the order of $G_2(\ell)$ is divisible by $\ell$, and $G_2(\ell)$ cannot be all of $\operatorname{SL}_2(\mathbb{F}_\ell)$, for otherwise the order of $G_2(\ell)/J_2$ would be too large with respect to the order of $G_1(\ell)/J_1$). If the order of $J_1$ is divisible by $\ell$, then $G_1(\ell)/J_1 \cong G_2(\ell)/J_2$ is cyclic, and this forces $G_2(\ell)$ to be either Borel or Cartan. We can take $H=G$.
\item exceptional: according to whether the projective image of $G_1(\ell)$ is isomorphic to either $A_4, A_5$ or $S_5$, there is a subgroup $H$ of $G$ with $[G:H]\bigm\vert 24$ such that $H_1(\ell)$ is cyclic, of order either $6$ or $10$. We are thus reduced to the Cartan case.
\item $\operatorname{SL}_2(\mathbb{F}_\ell)$: comparing orders, the isomorphism $G_1(\ell)/J_1 \cong G_2(\ell)/J_2$ forces $G_2(\ell) = \operatorname{SL}_2(\mathbb{F}_\ell)$. We can take $H=G$.
\end{itemize}
Next we claim that there is a subgroup $K$ of $H$, with $[H:K] \mid 4$, such that the kernel $\tilde{J_1}$ (resp. $\tilde{J_2}$) of $K(\ell) \to K_2(\ell)$ (resp. of $K(\ell) \to K_1(\ell)$) has order either $1$ or $\ell$. It is easily seen that this implies that $K$ satisfies property (3a). Notice that since the orders of $\tilde{J_1}, \tilde{J_2}$ divide $2\ell$ this is equivalent to asking that $-\operatorname{Id} \not \in \tilde{J_1}, \tilde{J_2}$. If neither $(-\operatorname{Id},\operatorname{Id})$ nor $(\operatorname{Id},-\operatorname{Id})$ is in $H$ we are done (simply take $K=H$), so let us consider the case when $H$ contains at least one of the two.
Again we need to distinguish two subcases, depending on the type of $H_1(\ell)$:
\begin{itemize}
\item split/nonsplit Cartan, Borel: the group $P:=\frac{H_1(\ell)}{N(H_1(\ell))} \times \frac{H_2(\ell)}{N(H_2(\ell))}$ is the product of two cyclic groups, so $P/2P$ has order dividing 4. Let $\overline{H}$ be the image of $H(\ell)$ in $P$, so that again $\overline{H}/2\overline{H}$ has order dividing 4. We set $K=\ker \left( H \to H(\ell) \to \overline{H} \to \overline{H}/2\overline{H} \right)$. It is clear that $[H:K] \mid 4$, and we now show that $(-\operatorname{Id},\operatorname{Id})$ and $(\operatorname{Id},-\operatorname{Id})$ do not belong to $K$. The argument is the same in the two cases, so we only consider the former. Clearly, it suffices to show that the image in $\overline{H}$ of $(-\operatorname{Id},\operatorname{Id})$ does not lie in $2\overline{H}$. Suppose by contradiction this is the case, take $p \in \overline{H}$ such that $2p$ is the class of $(-\operatorname{Id},\operatorname{Id})$ in $\overline{H}$, and fix an $h \in H(\ell)$ that maps to $p$. By construction, this is an element such that $h^2$ differs from $(-\operatorname{Id},\operatorname{Id})$ by elements of order $\ell$. In particular, $h^{2\ell}=(-\operatorname{Id},\operatorname{Id})$, so $h^\ell=(h_1,h_2)$ is an element of $H(\ell)$ which squares to $(-\operatorname{Id},\operatorname{Id})$. As the only square roots of $\operatorname{Id}$ in $\operatorname{SL}_2(\mathbb{F}_\ell)$ are $\pm \operatorname{Id}$, this provides us with an element $(h_1,\pm \operatorname{Id})$ of $H(\ell)$ such that $h_1$ has order 4. But this contradicts our initial result that if an element $(h_1,\pm \operatorname{Id})$ is in $G(\ell)$, then the prime-to-$\ell$ part of the order of $h_1$ is at most 2.
\item $\operatorname{SL}_2(\mathbb{F}_\ell)$: as we have already remarked, this forces $H_2(\ell)=\operatorname{SL}_2(\mathbb{F}_\ell)$. By \cite[Lemma 5.1]{MR1209248}, there exists an $M \in \operatorname{GL}_2(\mathbb{F}_\ell)$ and a character $\chi:H(\ell) \to \{\pm 1\}$ such that
\[
H(\ell)=\left\{ (x,y) \in \operatorname{SL}_2(\mathbb{F}_\ell)^2 \bigm\vert y=\chi(x,y) MxM^{-1}  \right\}.\]
We can take $K=\ker\left( H \to H(\ell) \xrightarrow{\chi} \{\pm 1\} \right)$, which has index at most 2 in $H$.
\end{itemize}

The group $K$ thus obtained has index dividing $[G:H][H:K] \bigm\vert 24 \cdot 4$ in $G$, and satisfies properties (3a) and (3b) (hence, as already remarked, (3d) and (3e)). We obtain a group $T$ as in (3) by setting $T=K$ if property (3c) is true for $K$; if this is not the case, we obtain a group $T$ as in (2) by setting $T=\operatorname{ker} \left(K \to K_1(\ell) \to K_1(\ell)/N(K_1(\ell)) \right)$: indeed in this case $[K:T] \bigm\vert 8$ and $T$ is pro-$\ell$ by construction.
\end{proof}

In light of the previous proposition and of theorem \ref{thm_Pink_GP}, theorem \ref{thm_PinkSL22} follows from the following statement:
\begin{proposition}\label{prop_Reduction2}
Let $\ell > 5$ be a prime number and $G$ an open subgroup of $\operatorname{SL}_2(\mathbb{Z}_\ell)^2$ satisfying properties (3a)-(3e) of proposition \ref{prop_FirstReduction}. For $i=1,2$ let $n_i$ be a positive integer such that $G_i$ contains $\mathcal{B}_\ell(n_i)$. At least one of the following holds:
\begin{enumerate}
\item $G'$ contains $\mathcal{B}_\ell(20 \max\{n_1,n_2\},20 \max\{n_1,n_2\})$
\item there exists a subgroup $T$ of $G$, of index dividing $12$, such that if the Lie algebra $\mathcal{L}(T)$ contains $\ell^k \mathfrak{sl}_2(\mathbb{Z}_\ell) \oplus \ell^k \mathfrak{sl}_2(\mathbb{Z}_\ell)$ for a certain positive integer $k$, then $T'$ contains $\mathcal{B}_\ell(p,p)$, where
$
p=2k+\max\left\{2k,8n_1,8n_2 \right\}.
$
\end{enumerate}
\end{proposition}


We now prove proposition \ref{prop_Reduction2} by analyzing the various possibilities for the type of $G$. Every case is dealt with in a separate section, whose title is of the form $(\text{type of }G_1(\ell),\text{type of }G_2(\ell))$. Note that since $N(\operatorname{SL}_2(\mathbb{F}_\ell))$ is trivial we have $G_1(\ell)=\operatorname{SL}_2(\mathbb{F}_\ell)$ if and only if $G_2(\ell)=\operatorname{SL}_2(\mathbb{F}_\ell)$; this helps exclude a few more possibilities for the type of $G$. Moreover, the statement of proposition \ref{prop_Reduction2} is symmetric in $G_1, G_2$, so we can further use symmetry arguments to reduce the number of cases (thus, for example, we only need consider one of the two cases (Borel, Nonsplit Cartan) and (Nonsplit Cartan, Borel)). Finally, note that we shall prove a slightly stronger statement: the index of $T$ in $G$ can be taken to divide 2.


\subsection{Cases (Nonsplit Cartan, Split Cartan) and (Nonsplit Cartan, Borel)}\label{sect_FirstCase}
The same idea works in both cases: consider an element $(a,b) \in G$ with $[a]$ of maximal order in $G_1(\ell)$; notice in particular that $[a]$ has order dividing $\ell+1$. The group $G(\ell)$ contains the element
$
([a],[b])^{\ell(\ell-1)}=([a]^{\ell(\ell-1)},[\operatorname{Id}]),
$
and the order of $[a]^{\ell(\ell-1)}$ is given by $\displaystyle \frac{\operatorname{ord}([a])}{\left(\ell-1, \operatorname{ord}([a]) \right)}=:m$. Notice that $m \geq 3$: indeed we have $\left(\ell-1, \operatorname{ord}([a]) \right) \leq 2$, so $m \leq 2$ would imply $\operatorname{ord} [a] \bigm\vert 4$, against the assumption that the order of $G_1(\ell)$ does not divide 8 (cf. proposition \ref{prop_FirstReduction} (3c)). Lemma \ref{lemma_NontrivialElementIsEasy} then shows that $G'$ contains $\mathcal{B}_\ell(20 \max\{n_1,n_2\},20 \max\{n_1,n_2\})$.

\subsection{Cases (Split Cartan, Split Cartan), (Borel, Borel) and (Split Cartan, Borel)}
We start by considering the type (Split Cartan, Split Cartan), the other two cases being essentially identical. Note that the groups $N(G_1(\ell))$ and $N(G_2(\ell))$ are trivial, so $G_1(\ell)$ and $G_2(\ell)$ are isomorphic by property (3a) of proposition \ref{prop_FirstReduction}, and we can find an element $(h_1,h_2) \in G$ such that $[h_1],[h_2]$ generate $G_1(\ell), G_2(\ell)$ respectively. By lemma \ref{lemma_Teichmuller1}, the limit $(g_1,g_2)$ of the sequence $(h_1,h_2)^{\ell^{2n}}$ satisfies $g_1^\ell=g_1, g_2^\ell=g_2$, and furthermore $[g_1]$, $[g_2]$ generate $G_1(\ell)$, $G_2(\ell)$ respectively. We choose bases in such a way that both $g_1$ and $g_2$ are diagonal. Write $g_i=\left( \begin{matrix} d_i & 0 \\ 0 & d_i^{-1} \end{matrix} \right)$, where $d_i$ satisfies $d_i^\ell=d_i$. Since $G_1(\ell) \cong G_2(\ell)$, we know that the orders of $[d_1]$ and $[d_2]$ agree; in particular, we can write $[d_2]=[d_1]^q$ for a certain integer $q$, $1 \leq q \leq \operatorname{ord} [d_1]$, that is prime to the order of $[d_1]$. We can, if necessary, apply on the second factor $\operatorname{SL}_2(\mathbb{Z}_\ell)$ the change of basis induced by the matrix $S=\left(\begin{matrix} 1 & 0 \\ 0 & -1 \end{matrix} \right)$: this change of basis exchanges the two diagonal entries of $g_2$, hence it allows us to assume $1 \leq q \leq \frac{\operatorname{ord} [d_1]}{2}$. 
Given that the Teichm\"uller lift is a homomorphism, we have $d_2=\omega([d_2])=\omega([d_1])^q=d_1^q$.
We now define the group $T$:
\begin{itemize}
\item if $[d_2]=-[d_1]$ or $[d_2]=-[d_1]^{-1}$, we let $T$ be the index-2 subgroup of $G$ defined by $\ker \left( G \to G_1(\ell) \to G_1(\ell)/2G_1(\ell)\right)$. If we now repeat the construction of $d_1, d_2$ for $T$, we find that $[d_2]=[d_1]$, and that the order of $T_1(\ell)$ does not divide 4 (recall that the order of $G_1(\ell)$ does not divide 8).
\item if $[d_2] \neq - [d_1]^{\pm 1}$ we set $T=G$.
\end{itemize}
Suppose that $\mathcal{L}(T)$ contains $\ell^k \mathfrak{sl}_2(\mathbb{Z}_\ell) \oplus \ell^k \mathfrak{sl}_2(\mathbb{Z}_\ell)$: we are going to show that $T'$ contains $\mathcal{B}_\ell(p,p)$ for $p=2k+8\max\{n_1,n_2\}$. Recall that $T$ contains an element $(g_1,g_2)$, where $g_i=\operatorname{diag} (d_i, d_i^{-1})$, and that we have set up our notation so that $d_2=d_1^q$ for some $1 \leq q \leq \frac{\operatorname{ord} [d_1]}{2}$.
Consider now the three matrices
\[
M_1=\left( \begin{matrix} 0 & 1 \\ 0 & 0 \end{matrix} \right),  M_2= \left( \begin{matrix} 1 & 0 \\ 0 & -1 \end{matrix} \right), M_3=\left( \begin{matrix} 0 & 0 \\ 1 & 0 \end{matrix} \right)
\]
and let $\pi_1$ (resp. $\pi_2, \pi_3$) be the linear map $\mathfrak{sl}_2(\mathbb{Z}_\ell) \to \mathbb{Z}_\ell \cdot M_i$ giving the projection of an element on its $M_1$ (resp. $M_2, M_3$) component. A $\mathbb{Z}_\ell$-basis of the Lie algebra $\mathfrak{sl}_2(\mathbb{Z}_\ell) \oplus \mathfrak{sl}_2(\mathbb{Z}_\ell)$ is given by $(M_i, 0), (0, M_j)$ for $i=1,2,3$ and $j=1,2,3$. Note that both $\mathcal{L}(T)$ and $\mathcal{L}(N(T))$ are stable under conjugation by $(g_1,g_2)$ (cf. \cite[Lemma 4.5]{AdelicEC}).
Writing elements of $\mathfrak{sl}_2(\mathbb{Z}_\ell) \oplus \mathfrak{sl}_2(\mathbb{Z}_\ell)$ as the 6-dimensional vectors of their coordinates in the basis just described, the action of conjugating by $(g_1,g_2)$ is given by
\[
(x_1,x_2,x_3,x_4,x_5,x_6) \mapsto (d_1^2 x_1,x_2,d_1^{-2} x_3,d_2^2 x_4,x_5, d_2^{-2} x_6).
\]
In particular, if we denote by $C$ the linear operator 
$
(x,y) \mapsto (g_1 x g_1^{-1},g_2 y g_2^{-1})
$
(acting on the Lie algebra $\mathfrak{sl}_2(\mathbb{Z}_\ell) \oplus \mathfrak{sl}_2(\mathbb{Z}_\ell)$),
we have
\[
\frac{1}{\ell-1}\sum_{i=0}^{\ell-2} C^i (x_1,x_2,x_3,x_4,x_5,x_6) = \left(0,x_2,0,0,x_4,0 \right)
\]
since $\displaystyle \frac{1}{\ell-1}\sum_{i=0}^{\ell-2} d_1^{2i} = \frac{1}{\ell-1} \frac{d_1^{2(\ell-1)}-1}{d_1^2-1}=0$ (recall that $d_1^\ell=d_1$ and $d_1^2 \neq 1$), and similarly for $d_1^{-2}$ and $d_2^{\pm 2}$. It follows that if $\mathcal{L}(T)$ or $\mathcal{L}(N(T))$ contains the vector $(x_1, \ldots, x_6)$, then it also contains the vector $(x_1,0,x_3,x_4,0,x_6)$ and, for every integer $j$, its image $(d_1^{2j} x_1,0,d_1^{-2j} x_3,d_2^{2j} x_4,0,d_2^{-2j} x_6)$ under conjugation by $(g_1,g_2)$ iterated $j$ times. Consider the matrix
$
V=\left( \begin{matrix} 1 & 1 & 1 & 1 \\ d_1^2 & d_1^{-2} & d_2^2 & d_2^{-2} \\ d_1^4 & d_1^{-4} & d_2^4 & d_2^{-4} \\ d_1^6 & d_1^{-6} & d_2^6 & d_2^{-6}  \end{matrix} \right):
$
this is a Vandermonde matrix, so its determinant is an $\ell$-adic unit as long as $d_1^2 \not \equiv d_1^{-2} \pmod \ell$, $d_2^2 \not \equiv d_2^{-2} \pmod \ell$ and $d_1^2 \not \equiv d_2^{\pm 2} \pmod{\ell}$. Recall that the orders of $[d_1]$, $[d_2]$ do not divide 4, so the first two conditions are automatically satisfied. Furthermore, under our assumptions the equality $[d_1]^2 = [d_2]^{\pm 2}$ implies $[d_1]=[d_2]$, that is, $q=1$. Consider first the case $q \neq 1$: the matrix $V$ is then invertible over $\mathbb{Z}_\ell$, hence the standard basis vectors $(1,0,0,0), (0,1,0,0), (0,0,1,0), (0,0,0,1)$ can be written as $\mathbb{Z}_\ell$-linear combinations of the rows of $V$. In turn, this implies that the vectors
\[
(x_1,0,0,0,0,0), (0,0,x_3,0,0,0),(0,0,0,x_4,0,0), (0,0,0,0,0,x_6)
\]
can be written as $\mathbb{Z}_\ell$-combinations of the four vectors $C^j (x_1,0,x_3,x_4,0,x_6)$ for $j=0,1,2,3$. Equivalently, we have shown that if $q \neq 1$ the Lie algebras $\mathcal{L}(T), \mathcal{L}(N(T))$ are stable under the projection operators $\pi_1 \oplus 0, \pi_3 \oplus 0, 0 \oplus \pi_1, 0 \oplus \pi_3$.
If, on the other hand, $q=1$ (so $g_1=g_2$), an even simpler computation shows that $\mathcal{L}(T), \mathcal{L}(N(T))$ are stable under the projection operators $\pi_1 \oplus \pi_1$ and $\pi_3 \oplus \pi_3$.
We have the immediate identities
\[
M_1 g_i = d_i^{-1} M_1, \; \pi_1(M_2 g_i)=0, \; \pi_1(M_3 g_i)=0, \; \pi_1(g_i)=0 \quad \text{ for } i=1,2,
\]
whence for any $A \in \operatorname{GL}_2(\mathbb{Z}_\ell)$ we have
$
\pi_1 \Theta_1(A g_i) = d_i^{-1} \pi_1(\Theta_1(A))$ for $i=1,2$.

Again we distinguish the cases $q \neq 1$ and $q=1$. In the former, we know that $(0 \oplus \pi_1)(\mathcal{L}(T))$ is contained in $\mathcal{L}(T)$, and that it is generated by an element of the form $(0 \oplus \pi_1)(\Theta_2(h))$, for a certain $h=(h_1,h_2) \in T$. We can choose an integer $m$ in such a way that $h \cdot (g_1,g_2)^m$ belongs to $N(T)$ (it suffices to choose $m$ such that $[h_1 g_1^m]=[\operatorname{Id}]=[h_2 g_2^m]$); for such an $m$, the element $\Theta_2(h \cdot (g_1,g_2)^m)$ lies in $\mathcal{L}(N(T))$, and since $\mathcal{L}(N(T))$ is stable under $0 \oplus \pi_1$ it also contains
\[
\begin{aligned}
d_2^m \cdot (0 \oplus \pi_1)\left(\Theta_2(h \cdot (g_1,g_2)^m)\right) & = d_2^m  (0, \pi_1 \Theta_1(h_2g_2^m))\\
& = (0,\pi_1(\Theta_1(h_2))).
\end{aligned}
\]
Now $\mathcal{L}(T)$ contains $(0 \oplus \pi_1)(\ell^k \mathfrak{sl}_2(\mathbb{Z}_\ell) \oplus \ell^k \mathfrak{sl}_2(\mathbb{Z}_\ell)) = 0 \oplus \ell^k\mathbb{Z}_\ell \cdot M_1$, so $(0,\ell^k M_1)$ is a multiple of $(0 \oplus \pi_1)(\Theta_2(h))$: the previous formula then shows that $\mathcal{L}(N(T))$ contains $0 \oplus \ell^k\mathbb{Z}_\ell \cdot M_1$. An analogous argument shows that $\mathcal{L}(N(T))$ contains $\ell^k\mathbb{Z}_\ell \cdot M_1 \oplus 0$: lemma \ref{lemma_NontrivialLieAlgebra} then implies that $T'$ contains
$
\mathcal{B}_\ell \left(2k+8 \max\left\{n_1,n_2\right\}, 2k+8 \max\left\{n_1,n_2\right\}\right).
$

\smallskip

Suppose on the other hand that $q=1$, that is, $d_1 = d_2$. Let us write, for the sake of simplicity, $g$ for $g_1=g_2$ and $d$ for $d_1=d_2$. As we have seen, both $\mathcal{L}(T)$ and $\mathcal{L}(N(T))$, thought of as subsets of $\mathfrak{sl}_2(\mathbb{Z}_\ell) \oplus \mathfrak{sl}_2(\mathbb{Z}_\ell)$, are stable under the maps $\pi_i \oplus \pi_i$ for $i=1,2,3$. Hence $\mathcal{L}(T)$ is the direct sum of three rank-2 subalgebras $R_i=(\pi_i \oplus \pi_i)(\mathcal{L}(T))$, $i=1,2,3$, with $R_i$ open in $\mathbb{Z}_\ell M_i \oplus \mathbb{Z}_\ell M_i$; similarly, $\mathcal{L}(N(T))$ is the direct sum of three algebras $S_i=(\pi_i \oplus \pi_i)(\mathcal{L}(N(T)))$.
We claim that $S_1=R_1$.  If $R_1$ is generated by the two elements $(\pi_1 \oplus \pi_1)\left(\Theta_2(w_1)\right)$ and $\left(\pi_1 \oplus \pi_1\right)\left(\Theta_2(w_2)\right)$ for some $w_1,w_2\in T$, then we can find integers $m_1, m_2$ such that $w_1(g,g)^{m_1}$ and $w_2(g,g)^{m_2}$ belong to $N(T)$. It follows that for $j=1,2$ the algebra $S_1$ contains
\[
d^{m_j} (\pi_1 \oplus \pi_1) (\Theta_2(w_j \cdot (g,g)^{m_j}))=d^{m_j} d^{-m_j} (\pi_1 \oplus \pi_1)( \Theta_2(w_j) ),
\]
i.e. $S_1=R_1$ as claimed.
Now notice that by assumption $\mathcal{L}(T)$ contains $\ell^k \mathfrak{sl}_2(\mathbb{Z}_\ell) \oplus \ell^k \mathfrak{sl}_2(\mathbb{Z}_\ell)$, so $R_1=(\pi_1 \oplus \pi_1)(\mathcal{L}(T))$ contains $\ell^k\mathbb{Z}_\ell \cdot M_1 \oplus \ell^k \mathbb{Z}_\ell \cdot M_1$, and the same is true for $(\pi_1 \oplus \pi_1)(\mathcal{L}(N(T)))$. As above, we conclude that $T'$ contains
$
\mathcal{B}_\ell \left(2k+8 \max\left\{n_1,n_2\right\}, 2k+8 \max\left\{n_1,n_2\right\}\right).
$

Finally, note that the (Borel, Borel) and (Split Cartan, Borel) cases are completely analogous: we simply need to choose for $(g_1,g_2)$ an element such that $[g_1]$ (resp. $[g_2]$) generates $T_1(\ell)/N(T_1(\ell))$ (resp. $T_2(\ell)/N(T_2(\ell))$).


\subsection{Case (Nonsplit Cartan, Nonsplit Cartan)}
We follow an approach very close to that of the previous section. Using lemma \ref{lemma_Teichmuller1} and the fact that $G(\ell)$ is the graph of an isomorphism $G_1(\ell) \to G_2(\ell)$, we can find an element $(g_1,g_2)$ of $G$ such that $g_i^{\ell^2}=g_i$ and $[g_i]$ generates $G_i(\ell)$; in a suitable basis we can write
$
g_i = \left( \begin{matrix} a_i & b_i \varepsilon_i \\ b_i & a_i \end{matrix} \right),
$
where $\varepsilon_i$ is an element of $\mathbb{Z}_\ell^\times \setminus \mathbb{Z}_\ell^{\times 2}$ (that is to say, $[\varepsilon_i]$ is not a square in $\mathbb{F}_\ell^\times$). The condition that the order of $G_i(\ell)$ does not divide 8 implies $a_ib_i \not \equiv 0 \pmod \ell$. 
For any $\ell$-adic unit $\varepsilon$ consider now the three matrices
\[
M_1(\varepsilon)=\left( \begin{matrix} 0 & \varepsilon \\ 1 & 0 \end{matrix} \right), \; M_2(\varepsilon)=M_2= \left( \begin{matrix} 1 & 0 \\ 0 & -1 \end{matrix} \right), \; M_3(\varepsilon)=\left( \begin{matrix} 0 & -\varepsilon \\ 1 & 0 \end{matrix} \right).
\]
A basis of $\mathfrak{sl}_2(\mathbb{Z}_\ell) \oplus \mathfrak{sl}_2(\mathbb{Z}_\ell)$ is given by
\[
(M_1(\varepsilon_1), 0), \; (M_2(\varepsilon_1), 0), \; (M_3(\varepsilon_1),0), \; (0, M_1(\varepsilon_2)), \; (0, M_2(\varepsilon_2)), \; (0, M_3(\varepsilon_2)),
\]
and again we write elements of $\mathfrak{sl}_2(\mathbb{Z}_\ell) \oplus \mathfrak{sl}_2(\mathbb{Z}_\ell)$ as six-dimensional vectors in this basis. Let $C$ be the linear operator (from $\mathfrak{sl}_2(\mathbb{Z}_\ell) \oplus \mathfrak{sl}_2(\mathbb{Z}_\ell)$ to itself) 
given by $(x,y) \mapsto (g_1,g_2)(x,y)(g_1,g_2)^{-1}$; once again, $\mathcal{L}(G)$ and $\mathcal{L}(N(G))$ are stable under $C$.
The matrix of $C$ in this basis is block-diagonal, the blocks being given by
$
\displaystyle B_i=\left(\begin{array}{ccc}
 1 & 0 & 0 \\
 0 & 1+2 \varepsilon_i  b_i^2 & 2a_i b_i \varepsilon_i  \\
 0 & 2a_i b_i & 1+2 \varepsilon_i  b_i^2
\end{array}\right).
$
Since the bottom-right $2$ by $2$ block of $B_i$ is simply $g_i^2$, the eigenvalues of $B_i$ are $1$ and the squares of the eigenvalues of $g_i$. The analogue of the condition $[d_1] \neq -[d_2]^{\pm 1}$ of the previous paragraph is `the only eigenvalue shared by $[B_1] \in \operatorname{GL}_3(\mathbb{F}_\ell)$ and $[B_2] \in \operatorname{GL}_3(\mathbb{F}_\ell)$ is 1'. Thus in this case we define the group $T$ as follows (notice that 1 is not an eigenvalue of $[g_i]^2$):
\begin{enumerate}
\item for an element $z \in \mathbb{Z}_\ell[\sqrt{\varepsilon_1},\sqrt{\varepsilon_2}]$ denote by $[z]$ its image in $\overline{\mathbb{F}_\ell}$. By construction we have $a_i \pm b_i \sqrt{\varepsilon_i} = \omega\left(\left[ a_i \pm b_i \sqrt{\varepsilon_i} \right] \right)$, where $\omega$ is the Teichm\"uller lift. If $[a_1 \pm \sqrt{\varepsilon_1}b_1]^2=[a_2 \pm \sqrt{\varepsilon_2}b_2]^2$, then (if necessary) we apply on the second factor $\operatorname{SL}_2(\mathbb{Z}_\ell)$ the change of basis induced by the matrix $S=\left(\begin{matrix} 1 & 0 \\ 0 & -1 \end{matrix} \right)$ to assume $[a_1 + \sqrt{\varepsilon_1}b_1]^2=[a_2 + \sqrt{\varepsilon_2}b_2]^2$. 
We then set
$
T:=\ker \left( G \to G(\ell) \to G(\ell)/2G(\ell) \right),
$
which (since $G(\ell)$ is cyclic) has index 2 in $G$. The element $(g_1,g_2)^2 \in T$ projects to a generator of $T(\ell)$, and we have
\[
(a_1+b_1\sqrt{\varepsilon_1})^2 = \omega\left([a_1+b_1\sqrt{\varepsilon_1}]^2 \right) = \omega\left([a_2+b_2\sqrt{\varepsilon_2}]^2 \right) = (a_2+b_2\sqrt{\varepsilon_2})^2,
\]
which -- using the fact that $a_i^2-\varepsilon_i b_i^2=\det g_i=1$ -- implies $\varepsilon_1 b_1^2=\varepsilon_2 b_2^2$.
 Notice that $T$, being of index 2 in $G$, is normal, hence its Lie algebra $\mathcal{L}(T)$ is stable not just under conjugation by elements of $T$, but also under conjugation by elements of $G$; the same is true for $\mathcal{L}(N(T))=\mathcal{L}(N(G))$. In particular, both these algebra are stable under conjugation by $(g_1,g_2)$, that is, they are stable under $C$. 
\item if $(a_1 \pm \sqrt{\varepsilon_1}b_1)^2$, $(a_2 \pm \sqrt{\varepsilon_2}b_2)^2$ are all distinct in $\overline{\mathbb{F}_\ell}$ we simply set $T=G$. In this case the squares of the eigenvalues of $g_1$ and of $g_2$ are distinct.
\end{enumerate}

We now assume that $\mathcal{L}(T)$ contains $\ell^k \mathfrak{sl}_2(\mathbb{Z}_\ell) \oplus \ell^k \mathfrak{sl}_2(\mathbb{Z}_\ell)$: we shall show that $T'$ contains $\mathcal{B}_\ell(2k+8\max\{n_1,n_2\},2k+8\max\{n_1,n_2\})$.

\smallskip

Suppose first that we are in subcase (2). Let $p_1(x)$ be the characteristic polynomial of $B_1$, and consider $p_1(C)$. This will be a block-diagonal operator whose first block is the null matrix and whose second block is of the form $\left(
\begin{array}{c|cc}
  0 & 0 \quad 0\\ \hline
  0 & \raisebox{-9pt}{{\LARGE\mbox{{$A$}}}} \\[-4pt]
  0 &
\end{array}
\right)$, with $A$ invertible modulo $\ell$ (this follows at once from the fact that the reduction modulo $\ell$ of this block can be computed as $p_1([B_2])$, and the only eigenvalue that is common to $[B_1]$ and $[B_2]$ is 1). Note furthermore that by the Hamilton-Cayley theorem the $2 \times 2$ identity can be expressed as a polynomial in $A$, so that ultimately the diagonal matrix with diagonal entries $(0,0,0,0,1,1)$
can be expressed a polynomial in $C$. Concretely, this is the operator
\[
\displaystyle \Pi:\left( \left( \begin{matrix}h_1 & x_1 \\ y_1 & -h_1 \end{matrix} \right), \left( \begin{matrix} h_2 & x_2 \\ y_2 & -h_2 \end{matrix} \right) \right) \mapsto \left( \left( \begin{matrix}0 & 0 \\ 0 & 0 \end{matrix} \right), \left( \begin{matrix} \displaystyle h_2 & \displaystyle  \frac{ x_2 - \varepsilon_2 y_2}{2} \\ \displaystyle  \frac{\varepsilon_2 y_2 - x_2}{2\varepsilon_2} & -h_2 \end{matrix} \right) \right),
\]
and we have just shown that $\mathcal{L}(T)$ and $\mathcal{L}(N(T))$ (being stable under $C$) are in particular also stable under $\Pi$.
As $T_2$ contains $\mathcal{B}_\ell(n_2)$, there exists an element $f_1$ of $T_1$ such that $\left(f_1, R\left(\ell^{n_2}\right) \right)$ belongs to $T$. Taking the $\ell(\ell^2-1)$-th power of this element shows that $N(T)$ contains the element
$
\left(f_1^{\ell(\ell^2-1)}, R\left( (\ell^2-1)\ell^{n_2+1}\right) \right),
$
and therefore $\mathcal{L}(N(T))$ contains
\[
\frac{1}{\ell^2-1} \Theta_2 \left(f_1^{\ell(\ell^2-1)}, R\left(  (\ell^2-1)\ell^{n_2+1} \right) \right) = \left( \frac{1}{\ell^2-1} \Theta_1\left(f_1^{\ell(\ell^2-1)} \right), \left( \begin{matrix} 0 & \ell^{n_2+1} \\ 0 & 0 \end{matrix} \right) \right).
\]
Applying $\Pi$ and multiplying by $2 \varepsilon_2$ we see that $\mathcal{L}(N(T))$ contains
$
\left(0,\left( \begin{matrix} 0 & \varepsilon_2 \ell^{n_2+1} \\ -\ell^{n_2+1} & 0 \end{matrix} \right) \right);
$
by lemma \ref{lemma_NontrivialLieAlgebra} we then find that $T'$ contains $\{\operatorname{Id} \} \times \mathcal{B}_\ell(10n_2+2)$. Swapping the roles of $T_1,T_2$ the same argument shows that $T'$ contains $\mathcal{B}_\ell(10n_1+2) \times \{\operatorname{Id} \} $, and we are done since this clearly implies that $T'$ contains $\mathcal{B}_\ell(20\max\{n_1,n_2\},20\max\{n_1,n_2\} )$.


Next consider subcase (1). 
Recall that algebras $\mathcal{L}(T)$ and $\mathcal{L}(N(T))$ are stable under $C$. 
We keep the notation $M_i(\varepsilon)$ from subcase (2), and we let $\pi_1(\varepsilon)$ (resp.~$\pi_2(\varepsilon), \pi_3(\varepsilon)$) be the linear maps $\mathfrak{sl}_2(\mathbb{Z}_\ell) \to \mathbb{Z}_\ell \cdot M_i(\varepsilon)$ giving the projection of an element on its $M_1(\varepsilon)$ (resp.~$M_2(\varepsilon), M_3(\varepsilon)$) component. Using the fact that $\varepsilon_1 b_1^2=\varepsilon_2 b_2^2$, one easily checks that
\[\displaystyle \pi_1(\varepsilon_1) \oplus \pi_1(\varepsilon_2) = - \frac{1}{4 \varepsilon_1 b_1^2} \left( \operatorname{Id}-2(1+2 \varepsilon_1 b_1^2 )C +C^2 \right),\]
from which we see that $\mathcal{L}(T), \mathcal{L}(N(T))$ are stable under $\pi_1(\varepsilon_1) \oplus \pi_1(\varepsilon_2)$ and therefore, by difference, also under
$
\tilde{\pi} : (x_1,x_2,x_3,x_4,x_5,x_6) \mapsto (0,x_2,x_3,0,x_5,x_6).
$
We now set $\Lambda$ to be the $6 \times 6$ matrix of $C$ in the basis $(M_1(\varepsilon_1),0),(M_2(\varepsilon_1),0),(M_3(\varepsilon_1),0)$, $(0,M_1(\varepsilon_2)),(0,M_2(\varepsilon_2)),(0,M_3(\varepsilon_2))$; as we have already seen in subcase (2),
this is the block-diagonal operator with blocks given by $1,g_1^2,1,g_2^2$.
We claim that $\tilde{\pi}(\mathcal{L}(T))$ and $\tilde{\pi}(\mathcal{L}(N(T)))$, seen as submodules of $\mathbb{Z}_\ell^6$, are stable under left multiplication by $\Lambda^{-1}$.
Indeed, since the Lie algebras of $T$ and of $N(T)$ are stable under conjugation by $(g_1,g_2)$ the claim follows from the identity
\begin{equation}
\tilde{\pi}\left( (g_1,g_2)^{-1} (t_1,t_2) (g_1,g_2) \right) = \Lambda^{-1} \cdot \tilde{\pi} \left( (t_1,t_2) \right) \quad \forall (t_1,t_2) \in \mathfrak{sl}_2(\mathbb{Z}_\ell)^2.
\end{equation}
Furthermore, 
one easily checks that, for all $t \in T$, we have
$
\tilde{\pi} \left( \Theta_2 \left( (g_1,g_2)^2\cdot t \right)\right) = \Lambda \cdot \tilde{\pi} \left( \Theta_2(t) \right).
$
Let now $w_1, \ldots, w_4\in T$ be such that $\tilde{\pi}(\mathcal{L}(T))$ is generated by $\tilde{\pi}(\Theta_2(w_1)), \ldots, \tilde{\pi}(\Theta_2(w_4))$. Since $[(g_1^2,g_2^2)]$ generates $T(\ell)$, for $i=1,\ldots, 4$ there is an integer $m_i$ such that $(g_1,g_2)^{m_i}w_i$ belongs to $N(T)$ (that is, it is trivial modulo $\ell$): it follows that $\Theta_2\left((g_1^2,g_2^2)^{m_i}w_i \right)$ is in $\mathcal{L}(N(T))$, and since $\mathcal{L}(N(T))$ is stable under both $\tilde{\pi}$ and $\Lambda^{-1}$ we find
\[
\Lambda^{-m_i} \cdot \tilde{\pi} \left( \Theta_2\left((g_1^2,g_2^2)^{m_i}w_i \right) \right) = \Lambda^{-m_i} \cdot \Lambda^{m_i} \cdot \tilde{\pi}(\Theta_2(w_i)) = \tilde{\pi} \left( \Theta_2(w_i) \right) \in \mathcal{L}(N(T)).
\]
This easily implies $\mathcal{L}(N(T)) \supseteq \tilde{\pi}\left(\mathcal{L}(N(T))\right) = \tilde{\pi}(\mathcal{L}(T)) \supseteq \tilde{\pi}(\ell^k \mathfrak{sl}_2(\mathbb{Z}_\ell) \oplus \ell^k \mathfrak{sl}_2(\mathbb{Z}_\ell))$. In particular, $\mathcal{L}(N(T))$ contains two elements $(u_1,0)$ and $(0,u_2)$ with $u_1, u_2 \not \equiv 0 \pmod{\ell^{k+1}}$:
by lemma \ref{lemma_NontrivialLieAlgebra} we conclude that $N(T)'$ contains
$
\mathcal{B}_\ell \left(2k+8 \max\left\{n_1,n_2\right\}, 2k+8 \max\left\{n_1,n_2\right\}\right).
$

\subsection{Case ($\operatorname{SL_2}(\mathbb{F}_\ell), \operatorname{SL_2}(\mathbb{F}_\ell)$)}\label{sect_LastCase}
In this case we take $T=G$, and assume that $\mathcal{L}(T)=\mathcal{L}(G)$ contains $\ell^k \mathfrak{sl}_2(\mathbb{Z}_\ell) \oplus \ell^k \mathfrak{sl}_2(\mathbb{Z}_\ell)$. We shall prove that $T'=G'$ contains $\mathcal{B}_\ell(4k,4k)$.
We consider the question of whether or not, for any given $m$, the group $G(\ell^m)$ is the graph of an isomorphism $G_1(\ell^m) \to G_2(\ell^m)$; the following lemma covers the case when this does \textit{not} happen:

\begin{lemma}\label{lemma_GraphSL2}
Let $m$ be a positive integer. Suppose $G$ contains an element of the form $(g_1,g_2)$, where $g_1$ is trivial modulo $\ell^m$ but $g_2$ is not (or symmetrically, $g_2$ trivial modulo $\ell^m$ and $g_1$ nontrivial). Then $G'$ contains $\mathcal{B}_\ell(4(m-1), 4(m-1))$.
\end{lemma}

\begin{proof}
Notice first that we necessarily have $m \geq 2$: indeed, $G(\ell)$ is the graph of an isomorphism $G_1(\ell) \to G_2(\ell)$, so if $g_1$ is trivial modulo $\ell$ then $g_2$ is trivial modulo $\ell$ as well. In particular, $(g_1,g_2)$ belongs to $\ker \left(G \to G(\ell)\right)=N(G)$. Now let
$
x_1=R(\ell)$, $y_1=L\left(\ell\right)$, $h_1=D\left( \ell\right)$.
By \cite[IV-23, Lemma 3]{SerreAbelianRepr}, since $G_1(\ell)=\operatorname{SL}_2(\mathbb{F}_\ell)$ we also have $G_1=\operatorname{SL}_2(\mathbb{Z}_\ell)$; in particular, we can find $x_2, y_2, h_2 \in G_2$ such that $x=(x_1,y_1)$, $y=(y_2,y_2)$ and $h=(h_1,h_2)$ all belong to $G$. As $G(\ell)$ is the graph of an isomorphism $G_1(\ell) \to G_2(\ell)$ and $x_1, y_1$ and $h_1$ are all trivial modulo $\ell$, the same must be true of $x_2,y_2,h_2$.
The elements $x^{\ell^{m-1}}, y^{\ell^{m-1}}$ and $h^{\ell^{m-1}}$ then satisfy
\begin{itemize}
\item their first coordinates topologically generate $\mathcal{B}_\ell(m)$, by lemma \ref{lemma_Generators} (3)
\item their second coordinates are trivial modulo $\ell^m$,
\end{itemize}
so the group they (topologically) generate contains an element of the form $(g_1^{-1},g_2')$, where $g_2'$ is trivial modulo $\ell^m$. Thus $G$ contains the product $g=(g_1^{-1},g_2')(g_1,g_2)=(\operatorname{Id},g_2'g_2)$, whose second coordinate is congruent to $g_2$ (and therefore nontrivial) modulo $\ell^m$. It is immediate to check that $g$ is in $N(G)$ (since its coordinates are trivial modulo $\ell$), hence $\mathcal{L}(N(G))$ contains $\Theta_2(g)$, which is of the form $(0,u)$ with $u$ nontrivial modulo $\ell^m$ (lemma \ref{lemma_NontrivialImpliesNontrivialAlgebra}). Applying lemma \ref{lemma_NontrivialLieAlgebra} we deduce that $N(G)$ contains $\left\{\operatorname{Id} \right\} \times \mathcal{B}_\ell(2(m-1))$ (in the notation of lemma \ref{lemma_NontrivialLieAlgebra} one can take $t=m-1$ and $n_2=0$).
To finish the proof, consider the group $H$ topologically generated by $x'=x^{\ell^{2m-3}}, y'= y^{\ell^{2m-3}}, h'=h^{\ell^{2m-3}}$, where the exponents make sense since we already remarked that $m \geq 2$. Denote by $H_1,H_2$ the projections of $H$ on the two components $\operatorname{SL}_2(\mathbb{Z}_\ell)$.
It is clear that $H_1 \supseteq \mathcal{B}_\ell(2(m-1))$ and $H_2 \subseteq \mathcal{B}_\ell(2(m-1))$, so the group generated by $H$ and $\left\{\operatorname{Id} \right\} \times \mathcal{B}_\ell(2(m-1))$ (which as we have seen is a subgroup of $G$) contains $\mathcal{B}_\ell(2(m-1),2(m-1))$, and we are done.
\end{proof}

%
%
%


We now show that for $m=k+1$ the hypothesis of the previous lemma is satisfied. Indeed suppose by contradiction that the projections
$
G(\ell^{k+1}) \to G_1(\ell^{k+1})$, $G(\ell^{k+1}) \to G_2(\ell^{k+1})
$
both have trivial kernel. Then Goursat's lemma implies that $G(\ell^{k+1})$ is the graph of an isomorphism $G_1(\ell^{k+1}) \to G_2(\ell^{k+1})$, i.e.~an automorphism of $\operatorname{SL}_2\left(\mathbb{Z}/\ell^{k+1}\mathbb{Z}\right)$. By \cite{MR840986} (see also \cite[Theorem 5(e)]{MR1670548}), and since $\ell > 5$, all such automorphisms are induced by conjugation, that is we can find a matrix $M \in \operatorname{GL}_2\left(\mathbb{Z}/\ell^{k+1}\mathbb{Z}\right)$ such that
$
G(\ell^{k+1}) = \left\{ (x,y) \in \operatorname{SL}_2\left(\mathbb{Z}/\ell^{k+1}\mathbb{Z}\right)^2 \bigm| y=MxM^{-1} \right\}.
$
Consequently, if we still denote by the same letter any lift of $M$ to $\operatorname{GL}_2(\mathbb{Z}_\ell)$, we have
\[
G \subseteq \left\{ (x,y) \in \operatorname{SL}_2(\mathbb{Z}_\ell)^2 \bigm| y \equiv MxM^{-1} \pmod{\ell^{k+1}} \right\}.
\]
Applying $\Theta_2$ and noticing that $\operatorname{tr}(MxM^{-1}) =  \operatorname{tr}(x)$ we deduce
\[
\mathcal{L}(G) \subseteq \left\{ (x,y) \in \mathfrak{sl}_2(\mathbb{Z}_\ell)^2 \bigm| y \equiv MxM^{-1} \pmod{\ell^{k+1}} \right\},
\]
but this contradicts the hypothesis that $\mathcal{L}(G)$ contains $\ell^{k} \mathfrak{sl}_2(\mathbb{Z}_\ell) \oplus \ell^{k} \mathfrak{sl}_2(\mathbb{Z}_\ell)$. Thus at least one of the two projections $G(\ell^{k+1}) \to G_i(\ell^{k+1})$ has nontrivial kernel, and the previous lemma shows that $G'$ contains $\mathcal{B}_\ell(4k, 4k)$.

\section{$\ell=2$, $n=2$}\label{sect_Pink2}
In this section we prove: 

\begin{theorem}\label{thm_PinkGL222}
Let $G$ be an open subgroup of $\operatorname{GL}_2(\mathbb{Z}_2)^2$ whose projection modulo 4 is trivial. Denote by $G_1, G_2$ the two projections of $G$ on the factors $\operatorname{GL}_2(\mathbb{Z}_2)$, and for $i=1,2$ let $n_i \geq 3$ be an integer such that $G_i$ contains $\mathcal{B}_2(n_i)$. Suppose furthermore that for every $(g_1,g_2) \in G$ we have $\det(g_1)=\det(g_2) \equiv 1 \pmod 8$.
If $\mathcal{L}(G)$ contains $2^k \mathfrak{sl}_2(\mathbb{Z}_2) \oplus 2^k \mathfrak{sl}_2(\mathbb{Z}_2)$ for a certain integer $k \geq 2$, then $G$ contains
$
\mathcal{B}_2(12(k+11n_2+5n_1+12)+1,12(k+11n_1+5n_2+12)+1).$
\end{theorem}

By an argument similar to that used for the case of odd $\ell$ (and that will be carried out at the end of this section) we can reduce the problem to the corresponding statement for $\operatorname{SL}_2(\mathbb{Z}_2)^2$:

\begin{theorem}\label{thm_PinkSL222}
Let $G$ be an open subgroup of $\operatorname{SL}_2(\mathbb{Z}_2)^2$ whose reduction modulo 4 is trivial. Denote by $G_1, G_2$ the two projections of $G$ on the factors $\operatorname{SL}_2(\mathbb{Z}_2)$, and for $i=1,2$ let $n_i \geq 3$ be an integer such that $G_i$ contains $\mathcal{B}_2(n_i)$. If $\mathcal{L}(G)$ contains $2^k \mathfrak{sl}(\mathbb{Z}_2) \oplus 2^k \mathfrak{sl}(\mathbb{Z}_2)$ for a certain integer $k \geq 2$, then $G$ contains
$
\mathcal{B}_2(6(k+11n_2+5n_1+12),6(k+11n_1+5n_2+12)).
$
\end{theorem}

The proof of this theorem, although technically involved, relies on a very simple idea: we can find an element of $G$ of the form $(\operatorname{Id},a)$, where $a$ is not too close to the identity $2$-adically, and this easily implies the conclusion by \cite[Theorem 5.2]{AdelicEC}, cf. lemma \ref{lemma_TrivialElementForZ2}. In order to find $a$ we proceed by contradiction: if there is no such $a$, then $G$ looks very much like the graph of a map $G_1 \to G_2$, and this imposes severe restrictions on its Lie algebra. Indeed, one can prove that in this case $\mathcal{L}(G)$ agrees to high $2$-adic order with the graph of a linear map $x \mapsto PxP^{-1}$. Quantifying this notion of `being $2$-adically very close to a graph' (cf. proposition \ref{prop_Conclusionl2}) gives a contradiction with the fact that $\mathcal{L}(G)$ contains $2^k \mathfrak{sl}(\mathbb{Z}_2) \oplus 2^k \mathfrak{sl}(\mathbb{Z}_2)$. 

\medskip

\noindent\textbf{Notation.} For the whole section, the symbols $G$, $G_1, G_2$, and $n_1, n_2$ will have the meaning given in the statement of theorem \ref{thm_PinkSL222}. We also set $x_1, y_1, h_1$ (resp.~$x_2, y_2, h_2$) to be $R(2^{n_1})$, $L(2^{n_1})$, $D(2^{n_1})$ (resp.~$R(2^{n_2})$, $L(2^{n_2})$, $D(2^{n_2})$): by construction, they are elements of $G_1$ (resp. of $G_2$).

\medskip

We start to deploy the strategy just outlined by showing that it is in fact enough to find an element $(\operatorname{Id},a)$ as above:

\begin{lemma}\label{lemma_TrivialElementForZ2}
Let $n \geq 3$. Suppose that $G$ contains an element of the form $(\operatorname{Id},a)$ (resp. $(a,\operatorname{Id})$), with $a$ nontrivial modulo $2^n$. Then $G$ contains
$
\mathcal{B}_2(6n+24n_2+n_1+18,6n+24n_2+18)
$
(resp. $\mathcal{B}_2(6n+24n_1+18,6n+24n_1+n_2+18)$.)
\end{lemma}

\begin{proof}
Consider the smallest normal subgroup $H$ of $G$ that contains $(\operatorname{Id},a)$. This is clearly of the form $\left\{\operatorname{Id}\right\} \times H_2$, where $H_2$ is the smallest normal subgroup of $G_2$ containing $a$. The Lie algebra $L$ of $H_2$ contains $\Theta_1(a)$, so it is nontrivial modulo $2^n$ by lemma \ref{lemma_NontrivialImpliesNontrivialAlgebra}. By normality of $H_2$ in $G_2$, $L$ is stable under conjugation by $\mathcal{B}_2(n_2)$, so lemma \ref{lemma_ConjStableSubspaces} says that $L$ contains $2^{n+4n_2+3} \mathfrak{sl}_2(\mathbb{Z}_2)$. Applying \cite[Theorem 5.2]{AdelicEC} we deduce that $H_2$ contains $\mathcal{B}_2(6n+24n_2+18)$, and we finish the proof as we did for lemma \ref{lemma_GraphSL2}.
\end{proof}
We now come to the hard part of the proof, namely showing that if no such $a$ exists, then $G$ is very close to being a graph. For any fixed integer $t \geq 3$, we distinguish two possibilities:

\begin{enumerate}
\item There exist two elements $(a,b)$ and $(a,b')$ of $G$ with $b \not \equiv b' \pmod{2^t}$, or equivalently, there exists an element of $G$ of the form $(\operatorname{Id},b'')$ with $b'' \not \equiv \operatorname{Id} \pmod{2^t}$. In this case we simply apply lemma \ref{lemma_TrivialElementForZ2}.
\item For every $a \in G_1$ there exists a (necessarily unique) $b \in G_2(2^t)$ such that, for every element of $G$ of the form $(a,c)$, we have $c \equiv b \pmod{2^t}$. In this case we write $b=\varphi(a)$, so that $\varphi$ is a well-defined function $G_1 \to G_2(2^t)$.
\end{enumerate}


As it is clear, the key step in proving theorem \ref{thm_PinkSL222} is to bound the values of $t$ for which this second case can arise. Let then $t \geq 3$ be an integer for which we are in case 2. We choose a function $\psi:G_1 \to G_2$ such that
\begin{itemize}
\item $\psi(a) \equiv \varphi(a) \pmod{2^t}$ for every $a\in G_1$;
\item $(a,\psi(a))$ belongs to $G$ for every $a \in G_1$.
\end{itemize}


As we shall see shortly, $\varphi$ is actually a continuous group morphism; on the other hand, $\psi$ does not necessarily have any nice group-theoretic properties, but allows us to work with well-defined elements of $\mathbb{Z}_2$ instead of congruence classes. We will also see that any such morphism $\varphi$ is, in a suitable sense, `inner', a fact that will lead to a contradiction for $t$ large enough. From now on, therefore, we work under the following assumption:


\begin{condition}\label{assumption_t} The integer $t \geq 3$ has the following property: for every $a \in G_1$ there exists a unique $b \in G_2(2^t)$ such that, for every element of $G$ of the form $(a,c)$, we have $c \equiv b \pmod{2^t}.$\end{condition}

To prove theorem \ref{thm_PinkSL222} we now proceed as follows. Exploiting (integrated forms of) the commutation relations $[h,x]=2x$ and $[h,y]=-2y$ we show that we can fix a basis in which $\psi(h_1)$ is diagonal (cf.~lemma \ref{lemma_hDiagonalizable}). We then prove that in this basis $\psi(x_1)$ and $\psi(y_1)$ are necessarily triangular up to some $2$-adically small error (proposition \ref{prop_ConjugationByN}), and making use of the relation $[x,y]=h$ we further show that $\psi(x_1),\psi(y_1)$ satisfy some additional constraints (which lead to proposition \ref{prop_ConjugationWeak}). This proves that, up to some $2$-adically small error, $G$ is contained in the graph of $x \mapsto PxP^{-1}$ for a suitable matrix $P$: the desired conclusion then follows easily by estimating the 2-adic valuation of the error terms.

\begin{lemma}
$\varphi$ defines a group morphism $G_1 \to G_2(2^t)$.
\end{lemma}

\begin{proof}
Let $a_1, a_2$ be two elements of $G_1$. Then $(a_1,\psi(a_1))(a_2,\psi(a_2))=(a_1a_2,\psi(a_1)\psi(a_2))$ belongs to $G$, so our assumption implies that $\psi(a_1)\psi(a_2) \equiv \varphi(a_1a_2) \pmod{2^t}$. As $\psi(a_1)$ (resp. $\psi(a_2)$) is congruent to $\varphi(a_1)$ (resp. $\varphi(a_2)$) modulo $2^t$ the claim follows.
\end{proof}


\begin{lemma}
$\varphi$ is continuous.
\end{lemma}
\begin{proof}
Denote by $\pi_1, \pi_2:\operatorname{SL}_2(\mathbb{Z}_2)^2 \to \operatorname{SL}_2(\mathbb{Z}_2)$ the projections on the two factors. As $x_1,y_1,h_1$ belong to $G_1$ we can find $a,b,c$ so that $x'=(x_1,a),y'=(y_1,b),h'=(h_1,c)$ all belong to $G$. Consider then $(x')^{2^t},(y')^{2^t},(h')^{2^t}$ and the group $H$ they generate (topologically). The projection $\pi_1(H)$ contains $x_1^{2^t}=R(2^{n_1+t})$ and $y_1^{2^t}=L(2^{n_1+t})$, hence it is open in $\operatorname{SL}_2(\mathbb{Z}_2)$ by an obvious variant of lemma \ref{lemma_Generators} (4). On the other hand, $\pi_2(H)$ is generated by $a^{2^t}, b^{2^t}, c^{2^t}$, so it is trivial modulo $2^t$. It follows that for any $g_1 \in H$ we have $\varphi(g_1)=1$, so $\ker \varphi$ is open and $\varphi$ is continuous.
\end{proof}

\begin{definition}
Let $g$ be an element of $\operatorname{SL}_2(\mathbb{Z}_2)$ (resp.~of a finite quotient $\operatorname{SL}_2(\mathbb{Z}/2^m\mathbb{Z})$) that is trivial modulo 4, and let $\beta$ be any 2-adic integer. Write $\beta=\sum_{n \geq 0} a_n 2^n$, where each $a_n$ is either 0 or 1. We set $g^\beta=\prod_{n \geq 0} g^{a_n 2^n}$, which is well-defined since for every finite $j$ only a finite number of terms appearing in the product are nontrivial modulo $2^j$.
\end{definition}

The continuity of $\varphi$ then implies
\begin{lemma}
Let $\beta$ be any $2$-adic integer and $g$ be an element of $G_1$. We have $\varphi(g^\beta)=\varphi(g)^{\beta}$.
\end{lemma}

Having dispensed with these necessary preliminaries we now begin with the proof proper. 
To find an element satisfying the assumptions of lemma \ref{lemma_TrivialElementForZ2} we look first at $(x_1,\psi(x_1))$:

\begin{lemma}\label{lemma_Psix1Vanishes}
If $\log \psi(x_1)$ vanishes modulo $2^{n_1+n_2}$, then $G$ contains
$
\mathcal{B}_2(30n_1+24,30n_1+n_2+24).
$
\end{lemma}

\begin{proof}
Exponentiating the hypothesis gives $\psi (x_1) \equiv \operatorname{Id} \pmod{2^{n_1+n_2}}$. There exist $a,b,c \in G_1$ such that $x'=(a,x_2),y'=(b,y_2),h'=(c,h_2)$ belong to $G$. Consider $(x')^{2^{n_1}},(y')^{2^{n_1}},(h')^{2^{n_1}}$: these three elements (topologically) generate a group $H$ such that $\pi_1(H)$ is trivial modulo $2^{n_1+1}$ (recall that $a,b,c$ are already trivial modulo 4) and $\pi_2(H)$ contains $\mathcal{B}_2(n_1+n_2)$ (lemma \ref{lemma_Generators}).

It follows that $H$ (hence $G$) contains an element of the form $(w,\psi(x_1)^{-1})$, where $w$ is trivial modulo $2^{n_1+1}$. Therefore $G$ contains the element $(x_1,\psi(x_1))(w,\psi(x_1)^{-1})=(x_1w,\operatorname{Id})$, where $x_1w \equiv x_1 \pmod {2^{n_1+1}}$ is nontrivial modulo $2^{n_1+1}$. The claim follows from lemma \ref{lemma_TrivialElementForZ2}.
\end{proof}

It is clear that if we are in the situation of the previous lemma theorem \ref{thm_PinkSL222} easily follows, so next we ask what happens if $\log \psi(x_1)$ is not too close to zero $2$-adically. For the sake of simplicity we set $\alpha=(1+2^{n_1})^2$. Note that $h_1x_1h_1^{-1}=x_1^\alpha$, so we have $\varphi(h_1)\varphi(x_1)\varphi(h_1)^{-1} = \varphi(x_1)^{\alpha}$, or equivalently
$
\psi(h_1)\psi(x_1)\psi(h_1)^{-1} \equiv \psi(x_1)^\alpha \pmod{2^t};
$
taking logarithms, which makes sense since both sides of the equation are trivial modulo 4, we obtain (via lemma \ref{lemma_CongruenceExponentials}, which we shall use from now on without further explicit mention)
\begin{equation}\label{eqn_LogarithmIsEigenvalue}
\psi(h_1) \cdot \log \psi(x_1) \cdot \psi(h_1)^{-1} \equiv \alpha \log \psi(x_1) \pmod{2^t}.
\end{equation}

Now this equation shows that the operator `conjugation by $\psi(h_1)$' admits $\log \psi(x_1)$ as an approximate eigenvector. If $\log \psi(x_1)$ is not too close to zero, this allows us to deduce properties of $\psi(h_1)$:

\begin{lemma}\label{lemma_hDiagonalizable}
With the notation of theorem \ref{thm_PinkSL222} and condition \ref{assumption_t}, let $U=t-3n_1-n_2-3$.
Suppose $\log \psi(x_1) \not \equiv 0 \pmod{2^{n_1+n_2}}$ and $U>3n_1$. Then $\psi(h_1)$ is diagonalizable over $\mathbb{Q}_2$, with eigenvalues $\lambda_1,\lambda_2 \in \mathbb{Z}_2$ that satisfy
$
\lambda_1\equiv 1+2^{n_1} \pmod{2^U}$, $\lambda_2\equiv \left(1+2^{n_1}\right)^{-1} \pmod{2^U}.
$
\end{lemma}

\begin{proof}
Denote by $\mathcal{C}_{\psi(h_1)}$ the linear endomorphism of $\mathfrak{sl}_2(\mathbb{Z}_2)$ given by conjugation by $\psi(h_1)$, and let $p(x)$ be its characteristic polynomial. Note that $\operatorname{tr} (\log \psi(x_1)) = \log \det \psi(x_1)=0$, so $\log \psi(x_1)$ is in $\mathfrak{sl}_2(\mathbb{Z}_2)$. Also let $\lambda_1, \lambda_2$ be the eigenvalues of $\psi(h_1)$.
An easy computation shows that $p(x)=(x-1)\left(x-\lambda_1^2\right)\left(x-\lambda_2^2\right)$ (the same result can also be deduced from the properties of the adjoint representation of $\mathfrak{sl}_2$).

With a little abuse of notation, in the course of the proof we shall use congruences (modulo powers of 2) that involve $\lambda_1,\lambda_2$: a priori, these might not be elements of $\mathbb{Z}_2$, so the precise meaning of these congruences is that we work with the ideals generated by the relevant powers of 2 in the ring of integers of $F$, where $F$ is a suitable quadratic extension of $\mathbb{Q}_2$ that contains $\lambda_1,\lambda_2$; likewise, we also extend $v_2$ to $F$.
By equation \eqref{eqn_LogarithmIsEigenvalue} we have
\[
\mathcal{C}_{\psi(h_1)} (\log \psi(x_1)) = \psi(h_1) \left( \log \psi(x_1) \right) \psi(h_1)^{-1}= \alpha \log \psi(x_1) + O(2^t),
\]
so that $\log \psi(x_1)$ is approximately an eigenvector for $\mathcal{C}_{\psi(h_1)}$. We deduce from lemma \ref{lemma_ApproximateEigenvalue1} and the assumption $\log \psi(x_1) \not \equiv 0 \pmod{2^{n_1+n_2}}$ that $p(\alpha) \equiv 0 \pmod {2^{t-n_1-n_2+1}}$. Since $\psi(h_1)$ is congruent to the identity modulo 4, it is easy to see that we have $\lambda_1 \equiv \lambda_2 \equiv 1 \pmod 4$, so $v_2(1+2^{n_1}+\lambda_i)=1$ for $i=1,2$. Hence $v_2(p(\alpha))$, which is given by
\[
v_2 \left((1+2^{n_1})^2-1\right)+v_2(1+2^{n_1}+\lambda_1)+v_2(1+2^{n_1}+\lambda_2)+v_2(1+2^{n_1}-\lambda_1)+v_2(1+2^{n_1}-\lambda_2),
\]
does not exceed
$
(n_1+1)+ 1 +1 +2\max_i v_2(1+2^{n_1}-\lambda_i),
$
so that
\[
\max_i v_2(1+2^{n_1}-\lambda_i) \geq \frac{v_2(p(\alpha)) -n_1-3}{2} \geq \frac{t-2n_1-n_2-2}{2}.
\]
Up to exchanging $\lambda_1$ and $\lambda_2$ we can assume the maximum is attained for $i=1$.

Set $U'= \left\lfloor \displaystyle \frac{t-2n_1-n_2-2}{2} \right\rfloor$. We have $\lambda_1 \equiv 1+2^{n_1} \pmod{2^{U'}}$ and, if $U' > 2n_1$ (a condition that is implied by the hypothesis $U>3n_1$), also $\lambda_2 = \lambda_1^{-1} \equiv 1-2^{n_1}+2^{2n_1} \pmod{2^{2n_1+1}}$. It follows in particular that $v_2(1+2^{n_1}-\lambda_2)=n_1+1$, so that we can improve our previous estimate to
\[
v_2(1+2^{n_1}-\lambda_1) \geq v_2(p(\alpha))-(n_1+1)-1-1-(n_1+1) \geq t-3n_1-n_2-3,
\]
that is $\lambda_1 \equiv \lambda_2^{-1} \equiv 1+2^{n_1} \pmod{2^U}$. Now since $U>3n_1$ the trace of $\psi(h_1)$ is given by
\[
\lambda_1+\lambda_2 = 1+2^{n_1} + 1-2^{n_1}+2^{2n_1}+O(2^{3n_1}),
\]
so
\[
\begin{aligned}
\operatorname{tr}(\psi(h_1))^2-4\det(\psi(h_1)) & =\left(2+2^{2n_1}+O(2^{3n_1})\right)^2-4 \\ 
& = 2^{2n_1+2} \left( 1+ O\left(2^{n_1} \right) \right)
\end{aligned}
\]
is a square in $\mathbb{Z}_2$ (since $n_1 \geq 3$). It follows that the eigenvalues of $\psi(h_1)$ lie in $\mathbb{Z}_2$, because 
$
\displaystyle \lambda_{1,2} = \frac{\operatorname{tr}(\psi(h_1)) \pm \sqrt{\left(\operatorname{tr}\psi(h_1)\right)^2-4}}{2}
$
is in $\mathbb{Q}_2$ (as the expression under square root is a square) and is $2$-integral (as $p(x)$ is monic with $2$-integral coefficients). It follows that $\psi(h_1)$ is diagonalizable over $\mathbb{Q}_2$, and that its eigenvalues satisfy the given congruences. 
\end{proof}

We now start building a base-change matrix $P$ such that $G$ is contained (up to error terms of large 2-adic valuation) in the graph of $x \mapsto PxP^{-1}$. As a first approximation, the following corollary yields a matrix $N$ such that $N^{-1}\psi(h_1)N$ is congruent to $h_1$ modulo a large power of 2:

\begin{corollary}\label{cor_ConstructionOfN}
Under the hypotheses of the previous lemma, there exists a $2$-integral matrix $N \in \operatorname{GL}_2(\mathbb{Q}_2)$ that satisfies:
\begin{enumerate}
\item $N^{-1}\psi(h_1)N$ is diagonal (with diagonal entries $\lambda_1,\lambda_2$ as above);
\item $v_2 (\det(N)) \leq n_1+1$.
\end{enumerate}
\end{corollary}

\begin{proof}
Let $w_1,w_2$ be two eigenvectors for $\psi(h_1)$, associated resp. with $\lambda_1,\lambda_2$, and chosen so as to be $2$-integral and to have at least one coordinate that is a $2$-adic unit. Let $N$ be the matrix having $w_1, w_2$ as columns: it is clear that $N$ satisfies (1). Now if $w_1,w_2$ are linearly independent over $\mathbb{F}_2$ we are done, for then $v_2(\det N)=0$; otherwise, up to rescaling $w_1,w_2$ and exchanging their two coordinates, we can assume they are of the form $w_1=\left(\begin{matrix} 1 \\ w_1' \end{matrix}\right), w_2=\left(\begin{matrix} 1 \\ w_2' \end{matrix}\right)$. The determinant of $N$ is simply $w_2'-w_1'$, hence we have
\[
\left( \begin{matrix} 0 \\ w_2'-w_1' \end{matrix} \right) \equiv 0 \pmod{\det(N)}  \quad  \Rightarrow  \quad  w_2 \equiv w_1 \pmod{\det(N)}.
\]
Applying $\psi(h_1)$ to both sides of the last congruence we find
$
\lambda_2 w_2 = \lambda_1 w_1 \pmod{\det(N)},
$
and comparing the first coordinates of these vectors we obtain $\lambda_1 \equiv \lambda_2 \pmod{\det(N)}$.
Since by the previous lemma we have $\lambda_1 \equiv 1+2^{n_1} \pmod{2^{2n_1}}$ and $\lambda_2 \equiv 1-2^{n_1} \pmod{2^{2n_1}}$, we have in particular $2^{n_1+1} \equiv 0 \pmod{\det(N)}$, whence the corollary.
\end{proof}


Assuming the hypotheses of lemma \ref{lemma_hDiagonalizable}, fix a matrix $N$ as in the previous corollary. We consider those elements $(g_1,\psi(g_1))$ of $G$ such that $N^{-1} \psi(g_1) N$ is $2$-integral:

\begin{lemma}\label{lemma_CanTakeLogs}
Assume that $\log \psi(x_1)$ does not vanish modulo $2^{n_1+n_2}$ and that $U>3n_1$, so that we can find an $N$ as above. Let $g_1$ be an element of $\mathcal{B}_2(2n_1+1) \subseteq G_1$. Then $N^{-1}\psi(g_1)N$ is $2$-integral and trivial modulo 4.
\end{lemma}
\begin{proof}
As $\mathcal{B}_2(2n_1+1)$ is generated by $L(2^{2n_1+1})$, $R(2^{2n_1+1})$ and $D(2^{2n_1+1})$ it is enough to show the conclusion of the lemma holds for these three elements. Since the proof is virtually identical in the three cases, we only consider $R(2^{2n_1+1})$. We have $R(2^{2n_1+1})=x_1^{2^{n_1+1}}$, hence \[\psi(R(2^{2n_1+1})) \equiv \psi(x_1)^{2^{n_1+1}} \pmod {2^t};\] as $\psi(x_1) \equiv \operatorname{Id} \pmod{4}$, the matrix $\psi(x_1)^{2^{n_1+1}}$ is congruent to the identity modulo $2^{n_1+3}$. Writing $\psi(x_1)^{2^{n_1+1}}$ as $\operatorname{Id} + 2^{n_1+3}B$ for a certain 2-integral matrix $B$ we have 
\[
N^{-1} \psi(R(2^{2n_1+1})) N = N^{-1}\left(\operatorname{Id}+2^{n_1+3}B\right)N = \operatorname{Id} + N^*  \left(  \frac{2^{n_1+3}}{\det(N)} B \right)  N,
\]
where $N^*=\det(N)N^{-1}$ is the adjugate matrix of $N$. Since $v_2 (\det N) \leq n_1+1$, this last expression is manifestly $2$-integral and congruent modulo $4$ to the identity.
\end{proof}

Let $N^*$ be the adjugate matrix of $N$ and $D=\left( \begin{matrix} \lambda_1 & \\ & \lambda_2 \end{matrix} \right)=N^{-1}\psi(h_1)N$. By the previous lemma, the following identity only involves 2-integral matrices:
\[
\begin{aligned}
\left(N^{-1}\psi(h_1)N\right) \, \left(N^{-1}\psi\left(R(2^{2n_1+1})\right)N\right) \, & \left(N^*\psi(h_1)^{-1}N\right)  = N^*\psi(h_1) \psi\left(R(2^{2n_1+1})\right) \psi(h_1)^{-1}N\\
                                                             & \equiv N^* \varphi(h_1) \varphi(R(2^{2n_1+1})) \varphi(h_1)^{-1} N\pmod{2^t} \\
																														 & \equiv N^* \varphi \left( R(2^{2n_1+1})^{\alpha} \right) N \pmod{2^t} \\
   																													 & \equiv N^* \varphi \left( R(2^{2n_1+1}) \right)^{\alpha} N \pmod{2^t}.
\end{aligned}
\]
Replacing $\varphi$ by $\psi$ and dividing through by $\det(N)$ we deduce
\begin{equation}\label{eqn_AxEigenvector}
D \left( N^{-1}\psi(R(2^{2n_1+1}))N \right) D^{-1} \equiv \left(N^{-1}\psi(R(2^{2n_1+1}))N\right) ^{\alpha} \pmod{2^{t-n_1-1}}.
\end{equation}
This equation forces $N^{-1}\psi(R(2^{2n_1+1}))N$ to be of a very specific form:

\begin{lemma}
There exists $c \in 4\mathbb{Z}_2$ such that $N^{-1}\psi(R(2^{2n_1+1}))N \equiv \left(\begin{matrix} 1 & c \\ 0 & 1 \end{matrix} \right) \pmod{2^{U-n_1-2}}$.
\end{lemma}
\begin{proof}
Note that since $t-n_1-1>U$ we can in particular rewrite equation \eqref{eqn_AxEigenvector} modulo $2^U$, and we know $D \equiv \left( \begin{matrix} 1+2^{n_1} & \\ & \left(1+2^{n_1}\right)^{-1} \end{matrix}\right) \pmod{2^U}$.
Recall that $N^{-1} \psi(R(2^{2n_1+1}))N$ is $2$-integral and trivial modulo 4 by lemma \ref{lemma_CanTakeLogs}, so its logarithm is well-defined. Since $D$ and $D^{-1}$ are both 2-integral, the same is true of $DN^{-1}\psi(R(2^{2n_1+1}))ND^{-1}$. Let $A_x=\log \left(N^{-1}\psi(R(2^{2n_1+1}))N\right)$ and write
$
A_x=\mu_x \left( \begin{matrix} 0 & 1 \\ 0 & 0\end{matrix} \right) + \mu_y \left( \begin{matrix} 0 & 0 \\ 1 & 0\end{matrix} \right) + \mu_h \left( \begin{matrix} 1 & 0 \\ 0 & -1\end{matrix} \right)
$
for certain $\mu_x,\mu_y,\mu_h \in \mathbb{Z}_2$. We claim that
$
\mu_h \equiv 0 \pmod{2^{U-n_1-1}}$ and $\mu_y \equiv 0 \pmod{2^{U-n_1-2}}.
$
Reducing equation \eqref{eqn_AxEigenvector} modulo $2^U$ and taking logarithms we get
\[
\alpha \cdot A_x \equiv D A_x D^{-1} \equiv \left( \begin{matrix} \lambda_1 & 0 \\ 0 & \lambda_2 \end{matrix}\right) A_x \left( \begin{matrix} \lambda_1 & 0 \\ 0 & \lambda_2 \end{matrix}\right)^{-1} \pmod{2^U},
\]
and the right hand side can be computed explicitly in terms of $\mu_x,\mu_h,\mu_y$. Using $\lambda_2=1/\lambda_1$ we find
\[
\alpha \cdot A_x \equiv  \lambda_1^2 \mu_x \left( \begin{matrix} 0 & 1 \\ 0 & 0\end{matrix} \right)+ \lambda_2^2 \mu_y \left( \begin{matrix} 0 & 0 \\ 1 & 0\end{matrix} \right) + \mu_h \left( \begin{matrix} 1 & 0 \\ 0 & -1\end{matrix} \right) \pmod{2^U},
\]
that is
\begin{equation}\label{eq_CongruencesLogX}
\begin{cases} \alpha \mu_x \equiv \lambda_1^2 \mu_x \pmod{2^U} \\ \alpha \mu_y \equiv \lambda_2^2 \mu_y \pmod{2^U} \\ \alpha \mu_h \equiv \mu_h \pmod{2^U}. \end{cases}
\end{equation}
Furthermore, one easily sees that $v_2(\alpha-{\lambda_2}^2)  =v_2\left( \left(1+2^{n_1} \right)^2 -\left( 1-2^{n_1}+O\left(2^{2n_1} \right) \right)^2 \right)=n_1+2$
and $v_2(\alpha-1)=n_1+1$.
Rewriting the last formula in \eqref{eq_CongruencesLogX} as $(\alpha-1)\mu_h \equiv 0 \pmod{2^U}$ shows that $\mu_h \equiv 0 \pmod{2^{U-n_1-1}}$, while $\alpha \mu_y \equiv \lambda_2^2 \mu_y \pmod{2^U}$ implies $\mu_y \equiv 0 \pmod{ 2^{U-n_1-2}}$. This proves that $A_x \equiv \left( \begin{matrix}0 & \mu_x \\ 0 & 0 \end{matrix} \right) \pmod {2^{U-n_1-2}}$, and exponentiating we find
\[
N^{-1} \psi(R(2^{2n_1+1}))N=\exp A_x \equiv \exp \left( \begin{matrix}0 & \mu_x \\ 0 & 0 \end{matrix} \right) \equiv \left( \begin{matrix}1 & \mu_x \\ 0 & 1 \end{matrix} \right) \pmod {2^{U-n_1-2}}.
\]
We can then take $c=\mu_x$, which is in $4\mathbb{Z}_2$ since $N^{-1} \psi(R(2^{2n_1+1}))N \equiv \operatorname{Id} \pmod 4$ by lemma \ref{lemma_CanTakeLogs}.
\end{proof}


A completely analogous argument yields similar results for $N^{-1}\psi(L(2^{2n_1+1}))N$, hence:
\begin{proposition}\label{prop_ConjugationByN}
Assume that
\begin{itemize}
\item $T:=U-n_1-2=t-4n_1-n_2-5$ is larger than $2n_1-2$ (that is, $U >3n_1$);
\item $G$ contains no element of the form $(\operatorname{Id},b)$, where $b \not \equiv \operatorname{Id} \pmod{2^t}$;
\item $\log \psi(x_1)$ does not vanish modulo $2^{n_1+n_2}$.
\end{itemize}

Then there exists a matrix $N \in \operatorname{GL}_2(\mathbb{Q}_2)$, with $2$-integral entries and whose determinant satisfies $v_2 (\det(N)) \leq n_1+1$, and scalars $c,d \in 4\mathbb{Z}_2$, such that
\[
N^{-1}\psi(h_1)N = \left( \begin{matrix} \lambda_1 &  \\ & \lambda_2 \end{matrix} \right) \equiv h_1 \pmod{2^T},
\]
\[
N^{-1}\psi\left(R(2^{2n_1+1})\right)N \equiv \left( \begin{matrix} 1 & c \\ 0 & 1 \end{matrix} \right) \pmod{2^T}, \quad N^{-1}\psi\left(L(2^{2n_1+1})\right)N \equiv \left( \begin{matrix} 1 & 0 \\ d & 1 \end{matrix} \right) \pmod{2^T}.
\]
\end{proposition}


\begin{remark}
As we shall see shortly, the product $cd$ is $2$-adically very close to $2^{4n_1+2}$, as one would expect. However, it is not true in general that $c,d$, taken separately, are $2$-adically very close to $2^{2n_1+1}$. Let us give an example of this phenomenon, which will also motivate the choices we make in the rest of the proof. Fix a positive integer $n_1 \geq 3$ and an integer $p \gg n_1$, and let $G$ be the group generated by $\mathcal{B}_2(p,p)$ and
$
\left\{ (g_1,g_2) \in \mathcal{B}_2(n_1)^2 \bigm\vert g_2=g_1 \right\}.
$
In our notation one can take $\psi$ to be the identity, and a matrix $N$ as in the statement of proposition \ref{prop_ConjugationByN} is given by $\left(\begin{matrix} 2^{n_1+1} & 0 \\ 0 & 1 \end{matrix} \right)$; notice that this matrix cannot be obtained from the construction of $N$ we gave in corollary \ref{cor_ConstructionOfN}, but its simple form makes it easier to make our point. In this situation we have
\[
N^{-1}\psi(R(2^{2n_1+1}))N = \left(
\begin{array}{cc}
 1 & 2^{n_1} \\
 0 & 1 \\
\end{array}
\right), \quad N^{-1}\psi(L(2^{2n_1+1}))N = \left(
\begin{array}{cc}
 1 & 0 \\
 2^{3 n_1+2} & 1 \\
\end{array}
\right),
\]
so $c=2^{n_1}$ and $d=2^{3n_1+2}$ are quite far $2$-adically from $2^{2n_1+1}$. 
\end{remark}

The parameters $c,d$ are up to now completely free, and they cannot be controlled by only using the relations $h_1x_1h_1^{-1}=x_1^\alpha$, $h_1^{-1}y_1h_1=y_1^{\alpha}$ (which are essentially integrated forms of the usual $\mathfrak{sl}_2$-Lie algebra relations $[h,x]=2x,[h,y]=-2y$). In order to say something meaningful about them, we shall need to use an integrated form of the Lie algebra relation $[x,y]=h$, that is to say we want to have some degree of control on commutators of the form $L(a)R(b)L(a)^{-1}R(b)^{-1}$. This is made possible by the following simple lemma, whose proof is immediate by induction:

\begin{lemma}\label{lemma_generators4}
\begin{itemize}
\item For any pair $(a,b)$ of elements of $\mathbb{Z}_2$ of valuation at least 1, the finite products
$
\displaystyle \prod_{i=-n}^{-1} R(a)^{(ab)^{-i}} \cdot \left(R(a)L(b)R(a)^{-1}L(b)^{-1}\right) \cdot \prod_{i=1}^n L(b)^{-(ab)^i}
$
converge, as $n \to \infty$, to $\left( \begin{matrix} \frac{1}{1-ab} & 0 \\ 0 & 1-ab \end{matrix} \right).$
\item
Let $(a,b)$ be as above and $(c,d)$ be any other pair of elements of $2$-adic valuation at least 1. The finite products
$
\displaystyle \prod_{i=-n}^{-1} R(c)^{(ab)^{-i}} \cdot \left(R(c)L(d)R(c)^{-1}L(d)^{-1}\right) \cdot \prod_{i=1}^n L(d)^{-(ab)^i}
$
converge to a limit for $n\to \infty$, and this limit is of the form $\left( \begin{matrix} \star & \star \\ \star & 1-cd \end{matrix} \right).$
\end{itemize}
\end{lemma}
Apply this lemma to $a=2^{2n_1+1}, b=2^{2n_1+1}$: the infinite product 
\[
\prod_{i=-\infty}^{-1} R(a)^{(ab)^{-i}} \cdot \left(R(a)L(b)R(a)^{-1}L(b)^{-1}\right) \cdot \prod_{i=1}^\infty L(b)^{-(ab)^i}
\]
converges to $\left( \begin{matrix} \displaystyle \frac{1}{1-2^{4n_1+2}} & 0 \\ 0 & 1-2^{4n_1+2} \end{matrix} \right) = h_1^\beta,$ where $\beta$ is defined by $(1+2^{n_1})^{-\beta} = 1-2^{4n_1+2}$. Applying $\varphi$ (which, being continuous, commutes with infinite products) we deduce that
\begin{equation}\label{eq_cd}
\varphi(h_1)^\beta = \prod_{i=-\infty}^{-1} \varphi(R(a))^{(ab)^{-i}} \cdot \left(\varphi(R(a))\varphi(L(b))\varphi(R(a))^{-1}\varphi(L(b))^{-1}\right) \cdot \prod_{i=1}^\infty \varphi(L(b))^{-(ab)^i}.
\end{equation}
By proposition \ref{prop_ConjugationByN} there exist $c,d \in 4\mathbb{Z}_2$ such that
\begin{equation}\label{eq_CongruencesAB}
\begin{aligned}
N^{-1}\psi(R(2^{2n_1+1}))N = N^{-1}\psi(R(a))N \equiv R(c)\equiv\left( \begin{matrix} 1 & c \\ 0 & 1 \end{matrix}\right) \pmod{2^T}, \\
N^{-1}\psi(L(2^{2n_1+1}))N =N^{-1}\psi(L(b))N \equiv L(d) \equiv \left( \begin{matrix} 1 & 0 \\ d & 1 \end{matrix}\right) \pmod{2^T}.
\end{aligned}
\end{equation}
Rewriting equation \eqref{eq_cd} in terms of $\psi$ and multiplying by $N$ (resp. $N^*$) on the right (resp.~left) we find
\[
N^*\psi(h_1)^\beta N \equiv N^* \prod_{i=-\infty}^{-1} \psi(R(a))^{(ab)^{-i}} \cdot \left[\psi(R(a)),\psi(L(b))\right] \cdot \prod_{i=1}^\infty \psi(L(b))^{-(ab)^i} N \pmod{2^{t}}.
\]
Further conjugating every term $\psi(R(a))$ and $\psi(L(b))$ by $N$, similarly to what we did in deriving equation \eqref{eqn_AxEigenvector}, using the congruences in \eqref{eq_CongruencesAB}, and dividing by $\det N$, we end up with
\[
\left(N^{-1}\psi(h_1)N\right)^\beta \equiv \prod_{i=-\infty}^{-1} R(c)^{(ab)^{-i}} \cdot \left(R(c)L(d)R(c)^{-1}L(d)^{-1}\right) \cdot \prod_{i=1}^\infty L(d)^{-(ab)^i} \pmod{2^{T-n_1-1}}.
\]
Applying the second part of the previous lemma to $R(c),L(d)$ we then obtain
\[
\begin{aligned}
h_1^\beta \equiv \left(N^{-1}\psi(h_1)N\right)^\beta & \equiv \prod_{i=-\infty}^{-1} R(c)^{(ab)^{-i}} \cdot \left(R(c)L(d)R(c)^{-1}L(d)^{-1}\right) \cdot \prod_{i=1}^\infty L(d)^{-(ab)^i} \\ & \equiv \left( \begin{matrix} \star & \star \\ \star & 1-cd \end{matrix} \right) \pmod{2^{T-n_1-1}},
\end{aligned}
\]
whence -- comparing the bottom-right coefficients -- we find $1-2^{4n_1+2} \equiv 1-cd \pmod{2^{T-n_1-1}}$. In particular, if $T \geq 5n_1+4$, we must have $v_2(c)+v_2(d)=4n_1+2$, and by symmetry we can assume that $v_2(c) \leq 2n_1+1$.

We deduce that $\displaystyle d \equiv \frac{2^{4n_1+2}}{c} \pmod{2^{T-n_1-1-v_2(c)}}$, and therefore
$
\displaystyle d \equiv \frac{2^{4n_1+2}}{c} \pmod{2^{T-3n_1-2}}.
$
Consider now the matrix $M=\left( \begin{matrix} 1 & 0 \\ 0 & 2^{2n_1+1}/c \end{matrix} \right)$ (which is $2$-integral, since $v_2(c) \leq 2n_1+1$). By construction we have $MR(2^{2n_1+1}) =R(c) M$, so that
\[
M R(2^{2n_1+1})  \equiv  R(c) M \equiv N^{-1}\psi(R(2^{2n_1+1}))NM \pmod{2^{T}},
\]
and furthermore -- since $N^{-1}\psi(h_1)N$ is diagonal and congruent to $h_1$ modulo $2^T$ -- we also have
\[
Mh_1 \equiv h_1M \equiv  N^{-1}\psi(h_1)N M \pmod{2^T}.
\]
Finally, using what we just proved on $d$ we find (for $T \geq 5n_1+4$)
\[
\begin{aligned}
M L(2^{2n_1+1}) & = \left( \begin{matrix} \displaystyle 1 & 0 \\ 2^{4n_1+2}/c & 2^{2n_1+1}/c \end{matrix} \right)
       \equiv \left( \begin{matrix} \displaystyle 1 & 0 \\ d & 2^{2n_1+1}/c \end{matrix} \right) \\ & \equiv L(d) M   
			\equiv N^{-1}\psi\left(L(2^{2n_1+1})\right)NM \pmod{2^{T-3n_1-2}}.
\end{aligned}
\]
Multiplying this last equation by $M^* = \left( \begin{matrix} 2^{2n_1+1}/c & 0 \\ 0 & 1 \end{matrix} \right)$ on the left we get
\[
\det M \cdot L(2^{2n_1+1}) \equiv M^*N^{-1}\psi\left(L(2^{2n_1+1})\right)N M \pmod{2^{T-3n_1-2}},
\]
and similar ones hold for $R(2^{2n_1+1})$, $h_1$. Given that $v_2(\det M) \leq 2n_1+1$, dividing by $\det M$ we find
 \[
(NM)^{-1} \psi \left(L(2^{2n+1})\right) NM \equiv L(2^{2n+1}) \pmod{2^{T-5n_1-3}},
\]
along with similar relations for $R(2^{2n_1+1}),h_1=D(2^{n_1})$. As $R(2^{2n_1+1}),L(2^{2n_1+1}),D(2^{n_1})$ generate $\mathcal{B}_2(2n_1+1)$ we have thus established
\begin{proposition}\label{prop_ConjugationWeak}
For every $g \in \mathcal{B}_2(2n_1+1)$ we have $(NM)^{-1}\psi(g)(NM) \equiv g \pmod{2^{T-5n_1-3}}$.
\end{proposition}

Notice that if we replace $NM$ by $\lambda NM$, for any $\lambda \in \mathbb{Q}_2$, then the previous proposition still holds, simply because the factors of $\lambda$ on the left hand side cancel out. We can then choose $\lambda$ in such a way that $\lambda NM$ is $2$-integral and has at least one coefficient which is a $2$-adic unit; we set $P:=\lambda NM$ for this value of $\lambda$.
We now give a version of proposition \ref{prop_ConjugationWeak} that applies to all of $G_1$:
\begin{proposition}
The following congruences hold for every $g \in G_1$:
\[
P^{-1}\psi(g)P \equiv g \pmod{2^{T-7n_1-2}} \text{ and } \operatorname{tr} \psi(g)= \operatorname{tr} g \pmod{2^{T-7n_1-2}}.
\]
\end{proposition}

\begin{remark}
As promised at the beginning of the section, this is essentially the statement that if $G$ does not contain an element of the form $(\operatorname{Id},b)$ with $b \not \equiv \operatorname{Id} \pmod{2^t}$, then $G$ is $2$-adically very close to being a graph.
\end{remark}

\begin{proof}
Clearly $v_2(\det P) \leq v_2(\det M)+v_2(\det N) \leq 3n_1+2$, and for all $g \in G_1$ and every $m \in \mathbb{N}$ we have
$
P^* \psi\left(g^m \right) P \equiv P^* \varphi \left(g^m \right) P \equiv P^* \varphi \left(g \right)^m P \equiv P^* \psi \left(g \right)^m P \pmod{2^t},
$
so dividing by $\det P$ we find $P^{-1} \psi\left(g^m \right) P \equiv P^{-1} \psi \left(g \right)^m P \pmod{2^{t-3n_1-2}}$.
Now $g$ is congruent to the identity modulo 4, hence $g^{2^{2n_1-1}}$ belongs to $\mathcal{B}_2(2n_1+1)$, so applying proposition \ref{prop_ConjugationWeak} and noticing that $T-5n_1-3 < t-3n_1-2$ we find
$
P^{-1}\psi\left( g \right)^{2^{2n_1-1}}P \equiv P^{-1}\psi\left( g^{2^{2n_1-1}} \right) P \equiv g^{2^{2n_1-1}} \pmod{2^{T-5n_1-3}}.
$
Taking logarithms we obtain
\[
2^{2n_1-1} P^{-1} \log \psi(g) P \equiv 2^{2n_1-1} (\log g) \pmod{2^{T-5n_1-3}},\]
whence
$
P^{-1} \log \psi(g) P \equiv \log g \pmod{2^{T-7n_1-2}}.
$
Since $\log g$ is trivial modulo 4, we can exponentiate both sides of the congruence to find
$
P^{-1}\psi(g)P \equiv g \pmod{2^{T-7n_1-2}},
$
as claimed. Taking the trace of this last congruence also gives $\operatorname{tr} \psi(g) \equiv \operatorname{tr}(g) \pmod{2^{T-7n_1-2}}$.
\end{proof}
Let now $(g_1,g_2)$ be an element of $G$. By the previous proposition we have
\[
P\left( g_1 -\frac{\operatorname{tr} g_1}{2} \operatorname{Id}  \right) \equiv \psi(g_1)P - \frac{\operatorname{tr}g_1}{2} P   \equiv \left( \psi(g_1) - \frac{\operatorname{tr} \psi(g_1)}{2} \operatorname{Id} \right)P   \pmod{2^{T-7n_1-3}},
\]
and since $\psi(g_1) \equiv g_2 \pmod{2^t}$ this implies $P\Theta_1(g_1) \equiv \Theta_1(g_2)P \pmod{2^{T-7n_1-3}}$.
Recalling that $\mathcal{L}(G)$ is the $\mathbb{Z}_2$-span of $\Theta_2(G)$, this implies that for every $(u_1,u_2) \in \mathcal{L}(G)$ we have the congruence $Pu_1 \equiv u_2P \pmod{2^{T-7n_1-3}}.$
%
Taking into account all the assumptions we made along the way, we have thus proved:
\begin{proposition}\label{prop_Conclusionl2}
Let $G$, $n_1,n_2$ be as in theorem \ref{thm_PinkSL222} and $t \geq 3$ be an integer. Assume:
\begin{enumerate}
\item $G$ contains no element of the form $(\operatorname{Id},b)$, where $b \not \equiv \operatorname{Id} \pmod{2^t}$;
\item $\log \psi(x_1)$ does not vanish modulo $2^{n_1+n_2}$;
\item $t-11n_1-n_2-8 \geq 0$ (that is, $T-7n_1-3 \geq 0$; this also implies the condition $U>3n_1$ of lemma \ref{lemma_hDiagonalizable}).




\end{enumerate}

Then for every $(u_1,u_2) \in \mathcal{L}(G)$ we have $Pu_1 \equiv u_2P \pmod {2^{t-11n_1-n_2-8}}$, where $P \in \operatorname{GL}_2(\mathbb{Q}_2)$ is 2-integral and has at least one coefficient which is a $2$-adic unit.
\end{proposition}


The result we were aiming for is now well within reach: 

\smallskip

\begin{proof}{(of theorem \ref{thm_PinkSL222})}
Suppose that $\mathcal{L}(G)$ contains $2^k \mathfrak{sl}_2(\mathbb{Z}_2) \oplus 2^k \mathfrak{sl}_2(\mathbb{Z}_2)$: then $\mathcal{L}(G)$ contains the elements $(u_1,u_2):=\left( 2^k \left(\begin{matrix} 0 & 1 \\ 0 & 0 \end{matrix} \right) ,0\right)$ and $(v_1,v_2):=\left( 2^k \left(\begin{matrix} 0 & 0 \\ 1 & 0 \end{matrix} \right) ,0\right)$. Since one of the coefficients of $P$ is a $2$-adic unit, at least one of the two congruences $Pu_1 \equiv u_2P \equiv 0 \pmod{2^{k+1}}$ and $Pv_1 \equiv v_2P \equiv 0 \pmod{2^{k+1}}$ does not hold. In particular, the conclusion of the previous proposition is \textit{false} if $t-11n_1-n_2-8=k+1$. Thus if we let $t = k+11n_1+n_2+9$ at least one of the three hypotheses must be false; for our choice of $t$, the inequality of condition (3) is satisfied, so either (1) or (2) must fail.
If (2) fails, then $G$ contains $\mathcal{B}_2(30n_1+24,30n_1+24+n_2)$ by lemma \ref{lemma_Psix1Vanishes}. On the other hand, if condition (1) is not satisfied, then lemma \ref{lemma_TrivialElementForZ2} implies that $G$ contains all of 
$
\mathcal{B}_2(6(k+11n_1+5n_2+12)+n_1,6(k+11n_1+5n_2+12))$, which is contained in $\mathcal{B}_2(30n_1+24,30n_1+24+n_2)$, and therefore contained in $G$ independently of which hypothesis (1) or (2) is the one that fails. Finally, note that the hypotheses of the theorem are symmetric in $G_1,G_2$, so we can repeat the whole argument switching the roles of $G_1, G_2$, which shows that $G$ also contains
$
\mathcal{B}_2(6(k+11n_2+5n_1+12),6(k+11n_2+5n_1+12)+n_2)
$
and concludes the proof of the theorem.
\end{proof}

As promised, we can finally prove theorem \ref{thm_PinkGL222}, by an argument similar to that used to deduce theorem \ref{thm_PinkSL22} from theorem \ref{thm_PinkGL22}:

\begin{proof}
We let $\det^*$ be as in the proof of theorem \ref{thm_PinkGL22}. Let $G^{\operatorname{sat}}=\{\lambda g \bigm\vert g \in G, \lambda \in 1+4\mathbb{Z}_2  \}$, denote by $U$ the intersection $G^{\operatorname{sat}} \cap \operatorname{SL}_2(\mathbb{Z}_2)^2$, and let $U_1,U_2$ be the two projections of $U$ on the factors $\operatorname{SL}_2(\mathbb{Z}_2)$. Notice that $U'=G'$: by assumption, for every $g \in G$ there exists $\lambda \in \mathbb{Z}_2$ such that $\det^* g=\lambda^2$; but then either $\lambda$ or $-\lambda$ is congruent to 1 modulo 4, so either $g/\lambda$ or $-g/\lambda$ belongs to $U$, and we can apply lemma \ref{lemma_Saturation}. 
Also remark that if $G_1$ contains $\mathcal{B}_2(n_1)$, then the same is true for $U_1$: indeed for every $g_1 \in \mathcal{B}_2(n_1)$ we know that there exists a certain $g_2 \in G_2$ such that $(g_1,g_2) \in G$. As $\det(g_2)$ equals $\det(g_1)=1$ by the assumptions on $G$, this shows that $(g_1,g_2)$ belongs to $U$ as well, and therefore $g_1$ belongs to $U_1$. The same argument obviously also works for $U_2$, and furthermore $U$ and $G$ have the same Lie algebra (again by lemma \ref{lemma_Saturation}). Finally, $U$ is trivial modulo 4 by construction. Applying theorem \ref{thm_PinkSL222} to $U$ we deduce that $U$ contains $\mathcal{B}_2(6(k+11n_2+5n_1+12),6(k+11n_1+5n_2+12))$, and therefore $G' = U'$ contains $\mathcal{B}_2(12(k+11n_2+5n_1+12)+1,12(k+11n_1+5n_2+12)+1)$ by lemma \ref{lemma_Commutator}.
\end{proof}

\section{Conclusion of the proof}\label{sect_RedTo2}
Using the results of the previous two sections we are now in a position to show theorem \ref{thm_GeneralPink} for all $n$.
Let us first remark that the inequality $[G:H] \leq 120$ appearing in the statement of \cite[Theorem 4.2]{AdelicEC} can be improved to $[G:H] \bigm\vert 24$: this follows immediately from the same proof and the simple remark that if $G(\ell)$ is exceptional, then $G$ contains a subgroup $H$, of index dividing 24, such that $H(\ell)$ is cyclic of order dividing 6 or 10. With this small improvement, case $n=1$ of theorem \ref{thm_GeneralPink} is (amply) covered by theorems 4.2 and 5.2 of \cite{AdelicEC}.
We now consider the case $n \geq 2$, starting with the odd primes. If $\ell=3$ we take $H=\operatorname{ker} \left(G \rightarrow \operatorname{SL}_2(\mathbb{F}_3)^n \right)$, for which the claim follows directly from the fact that it is pro-$\ell$ and from Pink's theorem (theorem \ref{thm_Pink_GP}).

We can then assume $\ell \geq 5, n \geq 2$. Denote by $\pi_i:\operatorname{SL}_2(\mathbb{Z}_\ell)^n \to \operatorname{SL}_2(\mathbb{Z}_\ell)$ the $n$ canonical projections and let $\pi_{i_1,i_2}=\pi_{i_1} \times \pi_{i_2}$ be the projection on the two factors numbered $i_1$ and $i_2$. For $m$ between $1$ and $n$, construct inductively groups $H_m$ as follows: apply \cite[Theorem 4.2]{AdelicEC} to $\pi_1(G)$ to find a subgroup $K_1$ of $\pi_1(G)$ of index dividing 24 and having property $(\star\star)$ in the notation of \cite{AdelicEC} (or, if $\pi_1(G)'=\operatorname{SL}_2(\mathbb{Z}_\ell)$, set $K_1=\pi_1(G)$), and set $H_1=\pi_1^{-1}(K_1)$: this is a subgroup of $G$ of index dividing 24. Assuming we have constructed $K_m$ and $H_m$, apply \cite[Theorem 4.2]{AdelicEC} to $\pi_{m+1}(H_m)$ to find a subgroup $K_{m+1}$ of index dividing 24 and having property $(\star\star)$, and set $H_{m+1}=\pi_{m+1}^{-1}(K_{m+1})$, again with the convention that $K_{m+1}=\pi_{m+1}(H_m)$ if $\pi_{m+1}(H_m)'=\operatorname{SL}_2(\mathbb{Z}_\ell)$. It is clear by construction that $H_n$ is an open subgroup of $G$ of index dividing $24^n$.
For $1 \leq t \leq n$ we also let $n_t$ be the minimal positive integer such that $\mathcal{B}_\ell(n_t)$ is contained in $\pi_t(H_n)$. Notice now that for all $m=1,\ldots,n-1$ the index $[H_m:H_{m+1}]$ is a divisor of 24, hence in particular coprime to $\ell \geq 5$. As $\mathcal{B}_\ell(n_t)$ is a pro-$\ell$ group, it follows that it is contained in $\pi_t(H_n)$ if and only if it is contained in $\pi_t(H_j)$ for all $j=1,\ldots,n$.
Now let $(i_1,j_1), (i_2,j_2), \ldots, (i_{n(n-1)/2},j_{n(n-1)/2})$ be the list of all $n(n-1)/2$ pairs $(i,j)$ with $i<j$, and construct inductively groups $H_{i_m,j_m}$ (for $m=1,\ldots,n(n-1)/2$) as follows. We start by formally setting $H_{i_0,j_0}=H_n$; then, assuming $H_{i_m, j_m}$ has been constructed, we apply theorem \ref{thm_PinkSL22} to $\pi_{i_{m+1},j_{m+1}}(H_{i_m,j_m})$ and we define $H_{i_{m+1},j_{m+1}}$ as follows:
\begin{enumerate}
\item if $\mathcal{B}_\ell(20\max\{n_{i_{m+1}},n_{j_{m+1}}\},20\max\{n_{i_{m+1}},n_{j_{m+1}}\})$ is contained in $\pi_{i_{m+1},j_{m+1}}(H_{i_m,j_m})$ we set $H_{i_{m+1},j_{m+1}}=H_{i_m,j_m}$ and $K_{i_{m+1},j_{m+1}}=\pi_{i_{m+1},j_{m+1}}(H_{i_m,j_m})$;
\item otherwise, there exists an open subgroup $K_{i_{m+1},j_{m+1}}$ of $\pi_{i_{m+1},j_{m+1}}(H_{i_m,j_m})$, of index dividing $48^2$ and having properties (2a) and (2b) of theorem \ref{thm_PinkSL22}, in which case we define $H_{i_{m+1},j_{m+1}}$ to be the inverse image in $H_{i_{m},j_{m}}$ of $K_{i_{m+1},j_{m+1}}$.
\end{enumerate}

Again, notice that the index $\left[H_{i_m,j_m}:H_{i_{m+1},j_{m+1}} \right]$ is prime to $\ell$, so that it follows inductively that for every $t=1,\ldots,n$ and every $m=1,\ldots,n(n-1)/2$ the open subgroup $\mathcal{B}_\ell(n_t)$ is contained in $\pi_t(H_{i_m,j_m})$. 
We finally set $H=H_{i_{n(n-1)/2},j_{n(n-1)/2}}$; by construction, it is a open subgroup of $G$ of index dividing $24^n 48^{n(n-1)}$. 
Denote by $\tau_i : \mathfrak{sl}_2(\mathbb{Z}_\ell)^n\to \mathfrak{sl}_2(\mathbb{Z}_\ell)$ (resp.~$\tau_{i,j}$) the projection on the $i$-th (resp.~$(i,j)$-th) factor. Now suppose that $\mathcal{L}(H)$ contains $\ell^k \mathfrak{sl}_2(\mathbb{Z}_\ell) \oplus \ldots \oplus \ell^k\mathfrak{sl}_2(\mathbb{Z}_\ell)$ for a certain $k>0$. We have
\[
\mathcal{L}(K_i) \supseteq \mathcal{L}(\pi_i(H)) = \tau_i (\mathcal{L}(H)) \supseteq \ell^k \mathfrak{sl}_2(\mathbb{Z}_\ell),
\]
so the properties of $K_i=\pi_i(H_i)$ given by \cite[Theorem 4.2]{AdelicEC} imply that it contains $\mathcal{B}_\ell(4k)$. It follows from the definition of the integers $n_t$ that we have $n_t \leq 4k$ for all $t$.
Consider now a pair of indices $i<j$. As before we have $\mathcal{L}(K_{i,j}) \supseteq \mathcal{L}(\pi_{i,j}(H)) = \tau_{i,j} (\mathcal{L}(H)) \supseteq \ell^k \mathfrak{sl}_2(\mathbb{Z}_\ell) \oplus \ell^k \mathfrak{sl}_2(\mathbb{Z}_\ell)$, so two cases arise (depending on whether we were in case (1) or (2) above):
\begin{enumerate}
\item either $\mathcal{B}_\ell(20\max\{n_i,n_j\},20\max\{n_i,n_j\}) \supseteq \mathcal{B}_\ell(80k,80k)$ is contained in $\pi_{i,j}(G)$,
\item or $K_{i,j}$ contains $\mathcal{B}_\ell(p,p)$ with $p = 2k+\max\left\{2k, 8n_i,8n_j \right\} \leq 34k$.
\end{enumerate}
Either way, we see that $\pi_{i,j}(G)$ contains $\mathcal{B}_\ell(80k,80k)$: 
once again, the index $[\pi_{i,j}(G):\pi_{i,j}(H)]$ is prime to $\ell$, so the fact that $\pi_{i,j}(G)$ contains $\mathcal{B}_\ell(80k,80k)$ implies that $\pi_{i,j}(H)$ contains $\mathcal{B}_\ell(80k,80k)$. Since this holds for every pair of indices $1 \leq i<j \leq n$, lemma \ref{lemma_IntegralGoursat} implies that $H$ contains $\displaystyle \prod_{i=1}^n \mathcal{B}_\ell\left(80(n-1)k\right)$, as claimed.

Finally consider the case $\ell=2$. Define $H$ to be the kernel of the reduction $G \to \operatorname{SL}_2(\mathbb{Z}/4\mathbb{Z})^n$. Suppose that $\mathcal{L}(H)$ contains $2^k \mathfrak{sl}_2(\mathbb{Z}_2) \oplus \ldots \oplus 2^k \mathfrak{sl}_2(\mathbb{Z}_2)$, and let $H_i=\pi_i(H), H_{i,j}=\pi_{i,j}(H)$.
Since $\mathcal{L}(\pi_i(H)) \supseteq 2^k \mathfrak{sl}_2(\mathbb{Z}_2)$, \cite[Theorem 5.2]{AdelicEC} implies that $H_i$ contains $\mathcal{B}_2(6k)$, and the integers $n_i$ appearing in the statement of theorem \ref{thm_PinkSL222} can all be taken to be $6k > 3$. Similarly, $\mathcal{L}(H_{i,j})$ contains $2^k \mathfrak{sl}_2(\mathbb{Z}_2) \oplus 2^k \mathfrak{sl}_2(\mathbb{Z}_2)$, hence by theorem \ref{thm_PinkSL222} the group $H_{i,j}$ contains $\mathcal{B}_2(582k+72, 582k+72)$: lemma \ref{lemma_IntegralGoursat} then implies that $H$ contains $\displaystyle \prod_{i=1}^n \mathcal{B}_2\left((n-1)(582k+73)\right)$. Finally, as $H$ is trivial modulo 4, it is clear that $k \geq 3$, so we have $582k + 73 \leq 607k$ and we are done. \qed

\medskip

\noindent\textbf{Acknowledgments.} I would like to thank my advisor, Nicolas Ratazzi, for his time and support, and the anonymous referee for the careful reading of the manuscript and the valuable comments.
I also acknowledge financial support from the ``Fondation Mathématique Jacques Hadamard".

\bibliography{Biblio}{}
\bibliographystyle{alpha}

\end{document}